\documentclass[11pt]{article}
\usepackage{amscd}
\usepackage{amsfonts}
\usepackage{amsmath}
\usepackage{amssymb}
\usepackage{amsthm}
\usepackage{bbm}
\usepackage{CJK}
\usepackage{fancyhdr}
\usepackage{graphicx}
\usepackage{indentfirst}
\usepackage{latexsym,bm}
\usepackage{mathrsfs}
\usepackage[arrow,matrix]{xy}
\usepackage[square, comma, sort&compress, numbers]{natbib}
\usepackage[colorlinks,linkcolor=red, anchorcolor=blue]{hyperref}
\allowdisplaybreaks[4]
\newtheorem{theorem}{Theorem}[section]
\newtheorem{lemma}[theorem]{Lemma}
\newtheorem{definition}[theorem]{Definition}
\newtheorem{proposition}[theorem]{Proposition}
\newtheorem{example}[theorem]{Example}

\newtheorem{corollary}[theorem]{Corollary}
\newtheorem{remark}[theorem]{Remark}
\newcommand\mapsfrom{\mathrel{\reflectbox{\ensuremath{\mapsto}}}}
\usepackage[top=1in,bottom=1in,left=1.25in,right=1.25in]{geometry}
\def\<{\langle}
\def\>{\rangle}
\def\a{\alpha}
\def\b{\beta}

\def\B{\Box}
\def\c{\cdot}

\def\d{\delta}
\def\D{\Delta}

\def\e{\eta}

\def\lr{\longrightarrow}

\def\m{\mapsto}
\def\o{\otimes}

\def\r{\rho}

\def\ra{\rightarrow}

\def\si{\sigma}
\def\ti{\times}

\def\tr{\triangleright}
\def\tl{\triangleleft}

\def\v{\varepsilon}
\def\vp{\varphi}

\def\z{\zeta}

\date{}
\begin{document}
\renewcommand{\baselinestretch}{1.2}
\renewcommand{\arraystretch}{1.0}
\title{\bf Four-angle Hopf modules for Hom-Hopf algebras}
 \date{}
\author {{\bf Xiaoqian Liu $^{a}$  \quad  Dongdong Yan$^{b}$\footnote {Corresponding author:  Ydd150365@163.com}
 \quad Xuchen Deng $^{a}$ \quad Danhua Wang $^{a}$}\\
{\small a: School of Information Engineering, Nanjing Xiaozhuang University, Nanjing }\\
{\small  Jiangsu 211171, P. R. of China}\\
{\small b: School of Mathematics and Physics, Nanjing Institute of Technology, Nanjing, }\\
{\small  Jiangsu 211167, P. R. of China}
}

 \maketitle
\begin{center}
\begin{minipage}{12.cm}
\noindent{\bf Abstract.}
In this paper, we introduce the notion of a four-angle Hopf module for a Hom-Hopf algebra $(H,\beta)$ and show that the category $\!^{H}_{H}\mathfrak{M}^{H}_{H}$
 of four-angle Hopf modules  is a monoidal category with either a Hom-tensor product $\otimes_{H}$ or a Hom-cotensor product $\Box_{H}$ as a monoidal product.
 We study the category $\mathcal{YD}^{H}_{H}$ of Yetter-Drinfel'd modules with bijective structure map can be organized as a braided monoidal category, in which we use a new monoidal structure. Finally, We prove an equivalence between the monoidal category $(~\!^{H}_{H}\mathfrak{M}^{H}_{H},\otimes_{H})$ or $(~\!^{H}_{H}\mathfrak{M}^{H}_{H},\Box_{H})$ of four-angle Hopf modules, and the monoidal category $\mathcal{YD}^{H}_{H}$ of Yetter-Drinfel'd modules,
   and furthermore, we give a braiding structure of the monoidal categorys $(~\!^{H}_{H}\mathfrak{M}^{H}_{H},\otimes_{H})$ (and  $(~\!^{H}_{H}\mathfrak{M}^{H}_{H},\Box_{H})$).  
   \\

\\

 \end{minipage}
 \end{center}
 \normalsize\vskip1cm

\section*{Introduction}
Hom-type structures play an important role in physics. In \cite{HLS06}, Hartwig et al. first introduced the Hom-Lie algebras to investigate the structures on some $q$-deformations of Witt and Virasoro algebras, in which the Jacobi identity is twisted by a endomorphism. In \cite{MS08}, Makhlouf and panaite gave the notation of Hom-associative algebra and extended usual functor between the categoried of Lie algebras and associative algebras to Hom-setting. In \cite{Y08}, Yau provided the construction of the free Hom-associative algebra and the enveloping algebra of a Hom-Lie algebra. Since then, Hom-analogues of various classical structures and results have been introduced and discussed by many authors (see \cite{AM10, BR23, CAB25, G10, S12}). All these generalizations coincide with the usual definitions when the structure map equals the identity.

The category of Yetter-Drinfel'd modules is one of the important categories of modules in the theory of Hopf algebras. The category is indeed a braided monoidal category under some suitable assumption. Via the braiding structures the notion of Yetter-Drinfel'd module plays an important role in the relations between knot theory and quantum group. A (right) Hopf module over Hopf algebra $H$ is a (right) $H$-module and a (right) $H$-comodule satisfying a compatibility condition which is very different from the one defining a right-right Yetter-Drinfel'd module. The definition of Hopf module is best understood by the fundamental theorem of Hopf modules (see \cite{S69}). In \cite{W89}, Woronowicz reinvented Hopf modules (and the fundamental theorem) to study differential calculi over quantum groups. In \cite{S94}, Schauenburg extended the fundamental theorem of Hopf modules to a monoidal equivalence between the category of two-sided two-cosided Hopf modules over $H$, and the category of Yetter-Drinfel'd module over $H$. Additionally, the equivalence is a monoidal one if we endow the category of Yetter-Drinfel'd structures with tensor product over $\Bbbk$, and the category of two-sided two-cosided Hopf modules with either the tensor product or the cotensor product (see \cite{D81}) over $H$. 
In \cite{ZW19}, Zhang and Wang proposed a new approach to braided monoidal categories by generalizing one of Schauenburg’s main results within the framework of Hopf quasigroups. Shortly thereafter, the dual result was established in the setting of Hopf coquasigroups (see \cite{GW20}), which is dual to \cite{ZW19}. In \cite{H25}, Han constructed Hopf bimodules and Yetter-Drinfeld modules over Hopf algebroids, as a generalization of the corresponding theory for Hopf algebras. In \cite{MP14}, Makhlouf and Panaite constructed the definition of Yetter-Drinfel'd module over Hom-bialgebra and show that the category of Yetter-Drinfel'd module is a quasi-braided pre-tensor category in two different ways.

A natural question to ask is whether the main theorem in \cite{S94} still holds in the setting of Hom-Hopf algebras.

This paper is organized as follows: In section 1, we recall the definitions and properties of Hom-type structures and braided monoidal category. Let $(H,\b)$ be a Hom-Hopf algebras. In section 2, we give the definition of four-angle Hopf module over $H$ and show that the category $\!^{H}_{H}\mathfrak{M}^{H}_{H}$ of four-angle Hopf modules over $H$ with bijective structure map has two structures of a monoidal category, the tensor structure of $(M,\z_{M}),(N,\z_{N})\in \!^{H}_{H}\mathfrak{M}^{H}_{H}$ being defined by $M \o_{H}N$ or $M \B_{H}N$. In section 3, we recall the notation of Yetter-Drinfel'd modules over $H$ and prove that the category $\mathcal{YD}^{H}_{H}$ is braided monoidal category, where the monoidal structure is redefined. In section 4, we first show that the equivalence of monoidal categories between the category of right-right Yetter-Drinfel'd modules over $H$ with bijective structure map, and the category of four-angle Hopf modules over $H$ with bijective structure map, with either $\o_{H}$ or $\B_{H}$ as product structure, which generalizes the main theorem in \cite{S94}. Finally, we define the braiding structure of the monoidal categorys $(~\!^{H}_{H}\mathfrak{M}^{H}_{H},\o_{H})$ (and $(~\!^{H}_{H}\mathfrak{M}^{H}_{H},\B_{H})$). 
\section{preliminaries}
\def\theequation{1.\arabic{equation}}
\setcounter{equation} {0}
A \emph{monoidal category} $\mathcal{C}=(\mathcal{C},\o ,\mathcal{I},a,l,r)$ is a category $\mathcal{C}$ equipped with a tensor product functor $\o:\mathcal{C}\ti \mathcal{C}\lr \mathcal{C}$, with a tensor unit object $\mathcal{I}\in \mathcal{C}$, with an associativity constraint isomorphism $a=a_{U,V,W}:(U\o V)\o W\lr U\o (V\o W)$ for any objects $U,V,W\in \mathcal{C}$, a left unit constraint $l=l_{U}:\mathcal{I}\o U\lr U$ and a right unit constraint $r=r_{U}:U\o\mathcal{I}\lr U$, for any object $U\in \mathcal{C}$, such that the pentagon axiom $a_{U,V,W\o X}\circ a_{U\o V,W,X}=(U\o a_{ V,W,X})\circ a_{U, V\o W,X}\circ (a_{U, V,W}\o X)$ and the triangle axiom $(U\o l_{V})\circ a_{U,\mathcal{I},V}=(r_{U}\o V)$ hold, for any objects $U,V,W,X\in \mathcal{C}$. A monoidal category $\mathcal{C}$ is \emph{strict} if all the constraints are identities.
\begin{definition} (\cite{K95}).
A braiding of a monoidal category $\mathcal{C}$ is a family of natural isomorphisms $c=c_{V,W}:V\o W\lr W\o V$ such that the following conditions hold
\begin{equation}\label{e1.1}
c_{U,V\o W}=a^{-1}_{V, W,U}\circ (V\o c_{U,W})\circ a_{V,U,W}\circ (c_{U,V}\o W)\circ a^{-1}_{U,V,W},
\end{equation}
\begin{equation}\label{e1.2}
c_{U\o V,W}=a_{ W,U,V}\circ (c_{U, W}\o V)\circ a^{-1}_{U,W,V}\circ (U\o c_{V,W})\circ a_{U,V,W},
\end{equation}
for any $U,V,W\in\mathcal{C}$, where $a$ is the associativity constraint in the monoidal category $\mathcal{C}$.
\end{definition}

Note that a \emph{braided monoidal category} is a monoidal category $\mathcal{C}$ with a braiding.

In the following, we will introduce the definitions and properties about Hom-type structures.
\begin{definition} (\cite{MP15}).
A Hom-algebra is a quadruple $(A, m,\eta,\z_{A})$, in which $A$ is a linear space, $\z_{A}:A\lr A$, $m:A\o A\lr A$ and $\eta:\Bbbk\lr A$ are linear maps, with notations $m(a\o a')=aa'$ and $\eta(1_{\Bbbk})=1_{A}$, such that, for any $a,b,c\in A$:
\begin{align*}
\begin{cases}
&\z_{A}(ab)=\z_{A}(a)\z_{A}(b),\quad \z_{A}(a)(bc)=(ab)\z_{A}(c),
\\&\z_{A}(1_{A})=1_{A},\quad1_{A}a=a1_{A}=\z_{A}(a).
\end{cases}
\end{align*}
We call $\z_{A}$ the structure map of $A$.
\end{definition}

A \emph{morphism of Hom-algebras} is a linear map $f:A\lr A'$ such that $\z_{A'}\circ f=f\circ \z_{A}$, $f(1_{A})=1_{A'}$ and $f(ab)=f(a)f(b)$, for any $a,b\in A$.
\begin{definition}\label{D1.3} (\cite{Y08A, Y09}).
Let $(A,\z_{A})$ be a Hom-algebra. A left $A$-module is a triple $(M, \z_{M},\a)$, in which $M$ is a linear space, $\z_{M}:M\lr M$ and $\a:A\o M\lr M$ are linear maps, with notation $\a(a\o m)=a\c m$, such that, for any $a\in A$ and $m\in M$:
\begin{align}
&\z_{A}(a)\c(b\c m)=(ab)\c\z_{M}(m),\label{e1.1.1}
\\&\z_{M}(a\c m)=\z_{A}(a)\c\z_{M}(m),\quad \label{e1.1.2}
\\&1_{A}\c m=\z_{M}(m).\label{e1.1.3}
\end{align}
\end{definition}

Similarly, we can define the right $A$-module. A \emph{morphism of left $A$-modules} is a linear map $f:M\lr N$ such that $\z_{N}\circ f=f\circ \z_{M}$ and $f(a\c m)=a\c f(m)$, for any $a\in A$ and $m\in M$.
\begin{definition} (\cite{MP15}).
A Hom-coalgebra is a quadruple $(C, \D,\v,\z_{C})$, in which $C$ is a linear space, $\z_{C}:C\lr C$, $\D:C\lr C\o C$ and $\v:C\lr \Bbbk$ are linear maps, with notation $\D(c)=c_{1}\o c_{2}$, such that, for any $c\in C$:
\begin{align*}
\begin{cases}
&\D(\z_{C}(c))=\z_{C}(c_{1})\o\z_{C}(c_{2}),\quad \z_{C}(c_{1})\o \D(c_{2})=\D(c_{1})\o \z_{C}(c_{2}),
\\&\v\circ \z_{C}=\v,\quad\v(c_{1})c_{2}=\z_{C}(c)=c_{1}\v(c_{2}).
\end{cases}
\end{align*}
\end{definition}

A \emph{morphism of Hom-coalgebras} is a linear map $f:C\lr C'$ such that $\z_{C'}\circ f=f\circ \z_{C}$, $\v_{C'}\circ f=\v_{C}$ and $\D_{C'}\circ f=(f\o f)\circ \D_{C}$.
\begin{definition} (\cite{MS08A, Y10}).
Let $(C, \z_{C})$ be a Hom-coalgebra. A left $C$-comodule is a triple $(M, \z_{M},\r,)$, in which $M$ is a linear space, $\z_{M}:M\lr M$ and $\r:M\lr C\o M$ are linear maps, with notation $\r(m)=m_{[-1]}\o m_{[0]}$, such that, for any $c\in C$ and $m\in M$:
\begin{align*}
\begin{cases}
&\r(\z_{M}(m))=\z_{M}(m_{[-1]})\o\z_{C}(m_{[0]}),\quad \v(m_{[-1]})m_{[0]}=\z_{M}(m),
\\&\z_{M}(m_{[-1]})\o [m_{[0][-1]}\o m_{[0][0]}]=[m_{[-1]1}\o m_{[-1]2}]\o \z_{C}(m_{[0]}).
\end{cases}
\end{align*}
\end{definition}

Similarly, we can define the right $C$-comodule. A morphism of left $C$-comodule is a linear map $f:M\lr N$ such that $\z_{N}\circ f=f\circ \z_{M}$ and $\r^{N}\circ f=(C\o f)\circ \r^{M}$.
\begin{definition} (\cite{MS08A, Y10}).
A Hom-bialgebra is a sextuple $(H,m,\eta, \D,\v,\b)$, in which $(H,m,\eta,\b)$ is a Hom-algebra and $(H,\D,\v,\b)$ is a Hom-coalgebra such that, for any $h,h'\in H$:
\begin{align*}
\begin{cases}
&\D(hh')=h_{1}h'_{1}\o h_{2}h'_{2},\quad\D(1_{H})=1_{H}\o 1_{H},
\\&\v(hh')=\v(h)\v(h'),\quad \v(1_{H})=1_{\Bbbk}.
\end{cases}
\end{align*}
\end{definition}
\begin{definition}(\cite{MS08A}).
A Hom-Hopf algebra is a Hom-bialgebra $(H,\b)$ endowed with a linear map $S:H\lr H$, called the antipode, such that, for any $h\in H$:
\begin{align*}
\begin{cases}
S(h_{1})h_{2}=\v(h)1_{H}=h_{1}S(h_{2}),
\\ S\circ \b=\b\circ S,
\end{cases}
\end{align*}
\end{definition}
\begin{remark}
As the consequences of the axiom above. For any $h,g\in H$, we have
\begin{align*}
\begin{cases}
&S(hg)=S(g)S(h),\quad S(1_{H})=1_{H},
\\& \D(S(h))=S(h_{2})\o S(h_{1}),\quad \v\circ S=\v.
\end{cases}
\end{align*}
\end{remark}
\begin{definition}
Let $(A,\z_{A})$ and $(B,\z_{B})$ be two Hom-algebras, $M$ a linear space and $\z_{M}:M\ra M$ a linear map. Then $(M,\z_{M})$ is called an $(A, B)$-bimodule if $(M,\z_{M})$ is a left $A$-module and a right $B$-module, with notations $A\o M\ra M, a\o m\m a\c m$ and $M\o B\ra M, m\o b\m m\c b$, such that, for any $a\in A$ and $b\in B$:
\begin{align}\label{e1.3}
\z_{A}(a)\c (m\c b)=(a\c m)\c \z_{A}(b).
\end{align}
\end{definition}
If $(B,\z_{B})=(A,\z_{A})$ as Hom-algebra, then $(M,\z_{M})$ defined above is called an $A$-bimodule (see \cite{MP15}).
\begin{definition}
Let $(C,\z_{C})$ and $(D,\z_{D})$ be two Hom-coalgebras, $M$ a linear space and $\z_{M}:M\ra M$ a linear map. Then $(M,\z_{M})$ is called an $(C, D)$-bicomodule if $(M,\z_{M})$ is a left $C$-comodule and a right $D$-comodule, with notations $\r^{l}_{M}: M\ra C\o M, m\m m_{[-1]}\o m_{[0]}$ and $\r^{r}_{M}: M\ra M\o D, m\m m_{(0)}\o m_{(1)}$, such that:
\begin{align}
\z_{C}(m_{[-1]})\o m_{[0](0)}\o m_{[0](1)}=m_{(0)[-1]}\o m_{(0)[0]}\o \z_{D}(m_{(1)}).
\end{align}
\end{definition}
If $(D,\z_{D})=(C,\z_{C})$ as Hom-coalgebra, then $(M,\z_{M})$ defined above is called an $C$-bicomodule (see \cite{LZ18}).

\section{Four-angle Hopf modules}
\def\theequation{2.\arabic{equation}}
\setcounter{equation} {0}
In this section, we introduce the concept of four-angle Hopf modules and equip the category of four-angle Hopf modules with two monoidal structure.
\begin{definition}
Let $(H,\b)$ be a Hom-Hopf algebra , $M$ a linear space and $\z_{M}:M\ra M$ a linear map. Then $(M,\z_{M})$ is called a left-left $H$-Hopf module if
\begin{enumerate}
\item[$(i)$] $(M, \z_{M}, \c)$ is a left $H$-module;
\item[$(ii)$] $(M, \z_{M},\r)$ is a left $H$-comodule;
\item[$(iii)$] the following compatibility condition holds
\begin{align}\label{e2.1}
(h\c m)_{[-1]}\o (h\c m)_{[0]}=h_{1} m_{[-1]}\o h_{2}\c m_{[0]},
\end{align}
\end{enumerate}
for any $h\in H$ and $m\in M$.
\end{definition}
As above, we can also define the left-right, right-left, and right-right $H$-Hopf modules as follows.
\begin{definition}\label{D2.3}
Let $(H,\b)$ be a Hom-Hopf algebra, $M$ a linear space and $\z_{M}:M\ra M$ a linear map. Then
\begin{enumerate}
 \item[$(1)$] $(M,\z_{M})$ is called a left-right $H$-Hopf module if $(M, \z_{M})$ is both a left $H$-module and a right $H$-comodule such that the compatibility condition holds:
    \begin{align}\label{e2.2}
(h\c m)_{(0)}\o (h\c m)_{(1)}=h_{1}\c m_{(0)}\o h_{2}m_{(1)},
    \end{align}
    for any $h\in H$ and $m\in M$.
\item[$(2)$] $(M,\z_{M})$ is called a right-left $H$-Hopf module if $(M, \z_{M})$ is both a right $H$-module and a left $H$-comodule such that the compatibility condition holds:
    \begin{align}\label{e2.3}
(m\c h)_{[-1]}\o (m\c h)_{[0]}= m_{[-1]}h_{1}\o m_{[0]}\c h_{2} ,
    \end{align}
    for any $h\in H$ and $m\in M$.
\item[$(3)$] $(M,\z_{M})$ is called a right-right $H$-Hopf module if $(M, \z_{M})$ is both a right $H$-module and a right $H$-comodule such that the compatibility condition holds:
    \begin{align}\label{e2.4}
  (m\c h)_{(0)}\o (m\c h)_{(1)}=m_{(0)}\c h_{1}\o m_{(1)}h_{2},
    \end{align}
    for any $h\in H$ and $m\in M$.

\end{enumerate}
\end{definition}
\begin{remark}
Note that Definition $\ref{D2.3}$ $(3)$ is different from monoidal Hom-Hopf module in \cite{CG11}.
\end{remark}
Let $(H,\b)$ be a Hom-Hopf algebra such that $\b$ is bijective. We can denote the category of Hopf modules by $\!^{H}_{H}\mathfrak{M}$, $\!_{H}\mathfrak{M}^{H}$, $\!^{H}\mathfrak{M}_{H}$ and $\mathfrak{M}^{H}_{H}$. Take $\mathfrak{M}^{H}_{H}$  whose objects are all right-right $H$-Hopf modules $(M,\z_{M})$ over $H$, with $\z_{M}$ bijective; the morphisms in the category are morphisms of right $H$-modules and a right $H$-comodules.
\begin{example}\label{E1}
Let $(H,\b)$ be a Hom-Hopf algebra.
\begin{enumerate}
\item[$(1)$] Let $V$ be a linear space and $\z_{V}:V\ra V$ a linear map. Then $(H\o V, \b\o \z_{V})$ is a left-left $H$-Hopf module with the following structures:
\begin{align*}
\begin{cases}
h\c(g\o v)=hg\o \z_{V}(v)
\\\r^{l}_{H\o V}(g\o v)=g_{1}\o g_{2}\o \z_{V}(v)
\end{cases}
\end{align*}
for any $h,g\in H$ and $v\in V$.
\item[$(2)$] $(H\o H, \b\o \b)$ is a left-left $H$-Hopf module with the following structure:
    \begin{align*}
\begin{cases}
h\c(g\o k)=hg\o \b(k)
\\\r^{l}_{H\o H}(g\o k)=g_{1}\o g_{2}\o \b(k)
\end{cases}
\end{align*}
for any $h,g,k\in H$.
\end{enumerate}
\end{example}
\begin{lemma}
Let $M$ be a linear space and $\z_{M}\in Aut(M)$ (called the set of all automorphisms of $M$). Then the following statements are equivalent:
\begin{enumerate}
\item[$(i)$] $(M,\z_{M})$ is an $H$-bimodule in $\!^{H}\mathfrak{M}$.
\item[$(ii)$] $(M,\z_{M})$ is a left $H$-comodule in $\!_{H}\mathfrak{M}_{H}$.
\item[$(iii)$] $(M,\z_{M})$ is an $H$-bimodule and a left $H$-comodule such that $(M,\z_{M})\in \!^{H}\mathfrak{M}_{H}$ and $(M,\z_{M})\in \!_{H}^{H}\mathfrak{M}$
\end{enumerate}
\end{lemma}
We call $(M,\z_{M})$ a two-sided $H$-Hopf module and denote by $\!_{H}^{H}\mathfrak{M}_{H}$ the category of these objects with morphisms are left and right linear and left colinear. In the same manner, we can define the category $\!_{H}\mathfrak{M}^{H}_{H}$.

Similarly, we can define the category $\!_{H}^{H}\mathfrak{M}^{H}$ (and $\!^{H}\mathfrak{M}^{H}_{H}$), whose objects are called two-cosided 
$H$-Hopf modules.
\begin{lemma}
Let $M$ be a linear space and $\z_{M}\in Aut(M)$. Then the following statements are equivalent:
\begin{enumerate}
\item[$(i)$] $(M,\z_{M})$ is an $H$-bimodule in $\!^{H}\mathfrak{M}^{H}$.
\item[$(ii)$] $(M,\z_{M})$ is an $H$-bicomodule in $\!_{H}\mathfrak{M}_{H}$.
\item[$(iii)$] $(M,\z_{M})$ is an $H$-bimodule and an $H$-bicomodule such that $(M,\z_{M})\in \!_{H}^{H}\mathfrak{M}, \!_{H}\mathfrak{M}^{H},$ $\!^{H}\mathfrak{M}_{H}, \mathfrak{M}_{H}^{H}$.
\end{enumerate}
\end{lemma}
We call $(M,\z_{M})$ a \emph{four-angle Hopf module} and denote by $\!_{H}^{H}\mathfrak{M}_{H}^{H}$ the category of these objects  with morphisms are linear and colinear on both sides. Also, There are some examples of four-angle Hopf module.
\begin{example}\label{E2.7}
Let $(H, \b)$ be a Hom-Hopf algebra such that the structure map $\b$ is bijective.
\begin{enumerate}
\item[$(1)$] $H_{a}:=H\o H$ is an object in $\!_{H}^{H}\mathfrak{M}_{H}^{H}$ with the following structures:
\begin{align*}
\begin{cases}
h\c (x\o y)=hx\o \b(y),\quad \r^{l}_{a}(x\o y)=\b^{-1}(x_{1}y_{1})\o x_{2}\o y_{2},
\\(x\o y)\c h=\b(x)\o yh,\quad \r^{r}_{a}(x\o y)=x_{1}\o y_{1}\o \b^{-1}(x_{2}y_{2}),
\end{cases}
\end{align*}
for any $x, y, h\in H$.
\item[$(2)$] $H_{b}:=H\o H$ is an object in $\!_{H}^{H}\mathfrak{M}_{H}^{H}$ with the following structures:
\begin{align*}
\begin{cases}
h\c (x\o y)=\b^{-1}(h_{1})x\o \b^{-1}(h_{2})y,\quad \r^{l}_{b}(x\o y)=x_{1}\o x_{2}\o \b(y),
\\(x\o y)\c h=x\b^{-1}(h_{1})\o y\b^{-1}(h_{2}),\quad\r^{r}_{b}(x\o y)=\b(x)\o y_{1}\o y_{2},
\end{cases}
\end{align*}
for any $x, y, h\in H$.
\item[$(3)$] $H_{c}:=H\o H$ is an object in $\!_{H}^{H}\mathfrak{M}_{H}^{H}$ with the following structures:
\begin{align*}
\begin{cases}
h\c (x\o y)=hx\o \b(y),\quad (x\o y)\c h=x\b^{-1}(h_{1})\o S\b^{-2}(h_{21})(\b^{-1}(y)\b^{-3}(h_{22})),
\\\r^{l}_{c}(x\o y)=x_{1}\o x_{2}\o \b(y),\quad \r^{r}_{c}(x\o y)=x_{1}\o y_{1}\o \b^{-1}(x_{2}y_{2}),
\end{cases}
\end{align*}
for any $x, y, h\in H$.
\item[$(4)$] $H_{d}:=H\o H$ is an object in $\!_{H}^{H}\mathfrak{M}_{H}^{H}$ with the following structures:
\begin{align*}
\begin{cases}
h\c (x\o y)=hx\o \b(y),\qquad~(x\o y)\c h=x\b^{-1}(h_{1})\o y\b^{-1}(h_{2}),
\\\r^{l}_{d}(x\o y)=x_{1}\o x_{2}\o \b(y),
\\\r^{r}_{d}(x\o y)=x_{1}\o \b^{-1}(y_{12})\o \b^{-1}(x_{2})(S\b^{-3}(y_{11})\b^{-2}(y_{2})),
\end{cases}
\end{align*}
for any $x, y, h\in H$.
\end{enumerate}
\end{example}

In the following, we will give two structures of a monoidal category of four-angle Hopf modules. First, one introduce the first structure. Let $(M,\z_{M})$ and $(N,\z_{N})$ be two four-angle Hopf modules over $H$ with bijective structure map. The Hom-tensor product $(M\o_{H}N, \z_{M}\o \z_{N})$ of $(M,\z_{M})$ and $(N,\z_{N})$ is defined by
\begin{align*}
M\o_{H}N:=\{m\o n\in M\o N ~| ~m\c h\o \z_{N}(n)=\z_{M}(m)\o h\c n, \forall ~h\in H\}.
\end{align*}
\begin{proposition}\label{P2.6}
Let $(H,\b)$ be a Hom-Hopf algebra such that $\b$ is bijective and assume that $(M,\z_{M})$ and $(N,\z_{N})$ are two four-angle Hopf modules over $H$, Then $(M\o_{H}N, \z_{M}\o \z_{N})\in \!_{H}^{H}\mathfrak{M}_{H}^{H}$ with the structures as follows:
\begin{align*}
h\c(m\o n)&=h\c m\o \z_{N} (n),
\\ \r^{l}(m\o n)&= \b^{-1}(m_{[-1]})\b^{-1}(n_{[-1]})\o m_{[0]}\o n_{[0]},
\\(m\o n)\c h&=\z_{M} (m)\o n\c h,
\\ \r^{r}(m\o n)&=m_{(0)}\o n_{(0)}\o \b^{-1}(m_{(1)})\b^{-1}(n_{(1)}),
\end{align*}
where $h\in H$ and $m\o n\in M\o_{H}N$.
\end{proposition}
\begin{proof}
First, we check that the actions stated are well defined. For any $h,g\in H$ and $m\o n\in M\o_{H}N$, we have
\begin{align*}
h\c(m\c g\o \z_{N}(n))&=h\c (m\c g)\o \z_{N}^{2} (n)
\\&=(\b^{-1}(h)\c m)\c\b( g)\o \z_{N}^{2} (n)
\\&=h\c \z_{M}(m)\o\b( g)\c \z_{N}  (n)
\\&=h\c(\z_{M}(m)\o g\c n)
\end{align*}
Thus $h\c(m\o n)\in M\o_{H}N$. Similarly, we have $(m\o n)\c h\in M\o_{H}N$.
It is easy to show that $M\o_{H}N$ is both an $(H,\b)$-bimodule and an $(H,\b)$-bicomodule.

Finally, we verify the four compatibility conditions Eq.$(\ref{e2.1})$, $(\ref{e2.2})$, $(\ref{e2.3})$ and $(\ref{e2.4})$. For any $h,g\in H$ and $m\o n\in M\o_{H}N$, we have
\begin{align*}
\r^{l}(h\c(m\o n))&=\r^{l}(h\c m\o \z_{N} (n))
\\&=\b^{-1}[(h\c m)_{[-1]}\b (n_{[-1]})]\o (h\c m)_{[0]}\o \z_{N}(n_{[0]})
\\&=\b^{-1}(h_{1}m_{[-1]})n_{[-1]}\o h_{2}\c m_{[0]}\o \z_{N}(n_{[0]})
\\&=h_{1}\b^{-1}(m_{[-1]}n_{[-1]}))\o h_{2}\c [m_{[0]}\o n_{[0]}]
\\&=h_{1}(m\o n)_{[-1]}\o h_{2}\c (m\o n)_{[0]}.
\\\r^{l}((m\o n)\c h)&=\r^{l}(\z_{M} (m)\o n\c h)
\\&=\b^{-1} [\b (m_{[-1]})(n\c h)_{[-1]})]\o \z(m_{[0]})\o (n\c h)_{[0]}
\\&=m_{[-1]}\b^{-1}[n_{[-1]}h_{1}]\o \z(m_{[0]})\o n_{[0]}\c h_{2}
\\&=\b^{-1}[m_{[-1]}n_{[-1]}]h_{1}\o (m_{[0]}\o n_{[0]})\c h_{2}
\\&=(m\o n)_{[-1]}h_{1}\o (m\o n)_{[0]}\c h_{2}.
\\\r^{r}(h\c(m\o n))&=\r^{r}(h\c m\o \z_{N} (n))
\\&=(h\c m)_{(0)}\o \z_{N}(n_{(0)})\o \b^{-1}[(h\c m)_{(1)}\b (n_{(1)})]
\\&=h_{1}\c  m_{(0)}\o \z_{N}(n_{(0)})\o \b^{-1}(h_{2}m_{(1)})n_{(1)}
\\&=h_{1}\c [m_{(0)}\o n_{(0)}]\o h_{2}(\b^{-1}(m_{(1)})n_{(1)})
\\&=h_{1}\c  (m\o n)_{(0)}\o h_{2}(m\o n)_{(1)}.
\\\r^{r}((m\o n)\c h)&=\r^{r}(\z_{M} (m)\o n\c h)
\\&=\z_{M}(m_{(0)})\o (n\c h)_{(0)}\o \b^{-1}[\b(m_{(1)})(n\c h)_{(1)}]
\\&=\z_{M}(m_{(0)})\o n_{(0)}\c h_{1}\o m_{(1)}\b^{-1}(n_{(1)}h_{2})
\\&=[m_{(0)}\o n_{(0)}]\c h_{1}\o   \b^{-1}(m_{(1)} n_{(1)})h_{2}
\\&=(m\o n)_{(0)}\c h_{1}\o  (m\o n)_{(1)}h_{2}.
\end{align*}

This completes the proof.
\end{proof}
\begin{proposition}\label{P2.7}
Let $(H,\b)$ be a Hom-Hopf algebra such that $\b$ is bijective and assume that $(M,\z_{M})$, $(N,\z_{N})$  and $(P,\z_{P})$ are three four-angle Hopf modules over $H$, with notation as above. Define the linear map, for any $m\in M$, $n\in N$ and $p\in P$,
\begin{align*}
\widetilde{a}_{M,N,P}:(M\o_{H} N)\o_{H} P\ra M\o_{H} (N\o_{H} P),\quad \widetilde{a}_{M,N,P}((m\o n)\o p)= m\o (n\o p).
\end{align*}
Then $\widetilde{a}_{M,N,P}$ is an isomorphism of $H$-modules and $H$-comodules on both sides.
\end{proposition}
The proof is straightforward.
\begin{proposition}\label{P2.8}
Let $(H,\b)$ be a Hom-Hopf algebra such that $\b$ is bijective and assume that $(M,\z_{M})$ is a four-angle Hopf module over $H$, Then as four-angle Hopf modules
\begin{align*}
M\o_{H} H\simeq M\simeq H\o_{H} M,
\end{align*}
where $(H,\b)$ be a four-angle Hopf module with its multiplication and comultiplication.
\end{proposition}
\begin{proof}
According to Proposition \ref{P2.6}, the structures of $M\o_{H} H$ is defined as follows
\begin{align*}
h\c(m\o g)&=h\c m\o \b(g),
\\ \r^{l}(m\o g)&= \b^{-1}(m_{[-1]}g_{1})\o m_{[0]}\o g_{2},
\\(m\o g)\c h&=\z_{M} (m)\o gh,
\\ \r^{r}(m\o g)&=m_{(0)}\o g_{1}\o \b^{-1}(m_{(1)}g_{2}),
\end{align*}
for any $m\in M$ and $h, g\in H$.

We define $\widetilde{r}:M\o_{H} H \ra M$ by $m\o g \m m\v(g)$ and $\widetilde{r}': M\ra M\o_{H} H $ by $m\m m\o 1_{H}$. First, one verifies that $\widetilde{r}\circ\widetilde{r}'=id$ and $\widetilde{r}'\circ\widetilde{r}=id$. For any $m\o g\in M\o_{H} H$, one have $m\o g=\z^{-1}(m)\c \b^{-1}(g)\o 1$. Thus
\begin{align*}
\widetilde{r}\circ\widetilde{r}'(m)=\widetilde{r}(m\o 1_{H})=m
\end{align*}
and
\begin{align*}
\widetilde{r}'\circ\widetilde{r}(m\o g)=\widetilde{r}'\circ\widetilde{r}(\z^{-1}(m)\c \b^{-1}(g)\o 1)=\widetilde{r}'(\z^{-1}(m)\c \b^{-1}(g))=m\o g
\end{align*}

Finally, it is easy to check that $\widetilde{r}$ is a map of $H$-modules and $H$-comodules on both sides. Thus $\widetilde{r}$ is an isomorphism of four-angle Hopf modules. Similarly, we can get $M\simeq H\o_{H} M$ by checking that the map
\begin{align*}
\widetilde{l}:H\o_{H} M\ra M, g\o m\m \v(g)m
\end{align*}
is an isomorphism.
\end{proof}

By the definition of monoidal category and Proposition \ref{P2.6}, \ref{P2.7} and \ref{P2.8}. we can easily obtain that the category $\!^{H}_{H}\mathfrak{M}^{H}_{H}$ is a monoidal category with the unit object $(H, \b)$.
\begin{proposition}\label{P2.9}
Let $(H,\b)$ be a Hom-Hopf algebra such that $\b$ is bijective. Then $(\!^{H}_{H}\mathfrak{M}^{H}_{H},$ $\o_{H},H,\widetilde{a}, \widetilde{r}, \widetilde{l})$ is a monoidal category.
\end{proposition}

Now, we describe the second structure of a monoidal category, which can be view as a dual of the first one. Let $(M,\z_{M})$ and $(N,\z_{N})$ be two four-angle Hopf modules over $H$ with bijective structure map. The Hom-cotensor product $(M\B_{H}N, \z_{M}\o \z_{N})$ of $(M,\z_{M})$ and $(N,\z_{N})$ is defined by
\begin{align}\label{e2.5}
M\B_{H}N:=\{m\o n\in M\o N ~| ~\r^{r}(m)\o \z_{N}(n)=\z_{M}(m)\o \r^{l}(n)\}.
\end{align}
As a dual of the result of Proposition \ref{P2.9}, by Lemma 2.3 and 2.4, we can get the following consequence.
\begin{proposition}\label{P1}
Let $(H,\b)$ be a Hom-Hopf algebra such that $\b$ is bijective. Then $(~\!^{H}_{H}\mathfrak{M}^{H}_{H},$ $\B_{H},H,\widehat{a},\widehat{l}, \widehat{r})$ is a monoidal category with the following structures: for any $m\in M$ and $h, g\in H$:
\begin{enumerate}
\item[$(1)$] Its tensor structure is defined by
\begin{align*}
&h\c(m\B n)=\b^{-1 }(h_{1})\c m\o \b^{-1 }(h_{2})\c n,&& \r^{l}(m\B n)= m_{[-1]}\o m_{[0]}\o \z_{N}(n),
\\&(m\B n)\c h=m\c \b^{-1 }(h_{1})\o n\c \b^{-1 }(h_{2}),&& \r^{r}(m\B n)= \z_{M}(m)\o n_{(0)}\o n_{(1)}.
\end{align*}
\item[$(2)$] Its unit object is $(H,\b)$.
\item[$(3)$] Its associativity constraint is
\begin{align*}
\widehat{a}_{M,N,P}:(M\B_{H} N)\B_{H} P\ra M\B_{H} (N\B_{H} P),\quad \widehat{a}_{M,N,P}((m\o n)\o p)= m\o (n\o p).
\end{align*}
\item[$(4)$] Its left unit constraint is
\begin{align*}
\widehat{l}:H\B_{H} M\ra M,\quad \widehat{l}(g\o m)= \v(g)m.
\end{align*}
\item[$(5)$] Its right unit constraint is
\begin{align*}
\widehat{r}:M\B_{H} H\ra M,\quad \widehat{r}(m\o g)= \v(g)m.
\end{align*}
\end{enumerate}
\end{proposition}
\section{Yetter-Drinfel'd modules for Hom-bialgebras}
\def\theequation{3.\arabic{equation}}
\setcounter{equation} {0}
In this section, we first recall the definition of Yetter-Drinfel'd modules over Hom-bialgebras. Then, one will give a new structure of a monoidal category of Yetter-Drinfel'd modules, which is different from the two forms defined as in \cite{MP14}.
\begin{definition} (see \cite{MP14})
Let $(H,\b)$ be a Hom-bialgebra, $V$ a linear space and $\z_{V}:V\ra V$ a linear map. Then $(V,\z_{V})$ is called a right-right Yetter-Drinfel'd module over $H$ if
\begin{enumerate}
\item[$(i)$] $(V, \z_{V},\tl)$ is a right $H$-module;
\item[$(ii)$] $(V, \z_{V},\r)$ is a right $H$-comodule;
\item[$(iii)$] the following compatibility condition holds
\begin{align}\label{e3.1}
(v\tl h_{2})_{(0)}\o \b^{2}(h_{1})(v\tl h_{2})_{(1)}=v_{(0)}\tl \b(h_{1})\o \b(v_{(1)})\b^{2}(h_{2}),
\end{align}
\end{enumerate}
for any $h\in H$ and $v\in V$.
\end{definition}
\begin{remark}
$(1)$ The compatibility condition is very different from the one defining a right-right $H$-Hopf module.

$(2)$ If $(H,\b)$ be a Hom-Hopf algebra, then Eq.(\ref{e3.1}) is equivalent to the following condition:
\begin{align*}
(v\tl h)_{(0)}\o (v\tl h)_{(1)}=v_{(0)}\tl \b^{-1}(h_{21})\o S(h_{1})(\b^{-1}(v_{(1)})\b^{-2}(h_{22})).
\end{align*}
\end{remark}
Let $(H,\b)$ be a Hom-Hopf algebra such that $\b$ is bijective. We denote by $\mathcal{YD}^{H}_{H}$ the category whose objects are all right-right Yetter-Drinfel'd modules $(V,\z_{V})$ over $H$, with $\z_{V}$ bijective; the morphisms in the category are morphisms of right $H$-modules and right $H$-comodules.
\begin{example}\label{E3.3}
$(1)$
A Hom-Hopf algebra $(H,\b)$ with antipode $S$ can be considered as a right-right Yetter-Drinfel'd module over itself with the comultiplication $\D_{H}$ as a right $H$-comodule, with the structure
\begin{align*}
x\leftharpoonup h=S\b^{-1}(h_{1})(\b^{-1}(x)\b^{-2}(h_{2})),\quad \forall ~x, h\in H
\end{align*}
as a right $H$-module, and denote it by $H_{A}=(H, \leftharpoonup, \D_{H}, \b)$.

$(2)$
A Hom-Hopf algebra $(H,\b)$ with antipode $S$ can be considered as a right-right Yetter-Drinfel'd module over itself with the multiplication $m_{H}$ as a right $H$-module, with the structure
\begin{align*}
\r(h)=\b^{-1}(h_{12})\o S\b^{-2}(h_{11})\b^{-1}(h_{2})\quad \forall ~h\in H
\end{align*}
as a right $H$-comodule, and denote it by $H_{B}=(H, m_{H}, \r, \b)$.
\end{example}
In the following, let $(H,\b)$ be a Hom-Hopf algebra such that $\b$ is bijective. We will show that the category $\mathcal{YD}^{H}_{H}$ of right-right Yetter-Drinfel'd modules is a braided monoidal category in a new way.
\begin{lemma}\label{L3.3}
Let $(V,\z_{V})$ and $(W,\z_{W})$ be two Yetter-Drinfel'd modules over $H$. Then $(V\o W, \z_{V}\o \z_{W})\in \mathcal{YD}_{H}^{H}$ with the structures as follows:
\begin{align*}
(v\o w)\tl h&=v\tl \b^{-1}(h_{1})\o w\tl \b^{-1}(h_{2}),
\\\r^{r}(v\o w)=&v_{(0)}\o n_{(0)}\o \b^{-1}(v_{(1)}w_{(1)}),
\end{align*}
for any $h\in H$, $v\in V$ and $w\in W$.
\end{lemma}
The proof is not hard and the interested readers can refer to the calculation of Proposition 3.5 in \cite{ZGW19}.
\begin{lemma}\label{L3.4}
Let $(u,\z_{U})$, $(V,\z_{V})$ and $(W,\z_{W})$ be three Yetter-Drinfel'd modules over $H$. Then the linear map
\begin{align*}
a_{U,V,W}:(U\o V)\o W\ra U\o (V\o W), \quad a_{U,V,W}((u\o v)\o w)=u\o (v\o w)
\end{align*}
is an isomorphism of right $(H,\b)$-modules and right $(H,\b)$-comodules.
\end{lemma}
The proof is straightforward.
\begin{theorem}\label{T3.5}
Let $(H,\b)$ be a Hom-Hopf algebra such that $\b$ is bijective. Then the category $\mathcal{YD}^{H}_{H}$ is a braided monoidal category, with tensor product $\o$ and associativity constraint $a$ defined as in Lemma \ref{L3.3} and \ref{L3.4}, respectively. Its braiding structure is defined by
\begin{align*}
c_{V,W}:V\o W\ra W\o V, \quad c_{V,W}(v\o w)=\z_{W}^{-1}(w_{(0)})\o \z_{V}^{-1}(v)\tl \b^{-2}(w_{(1)}),
\end{align*}
and the inverse
\begin{align*}
c_{V,W}^{-1}(w\o v)=\z^{-1}_{V}(v)\tl S^{-1}\b^{-2}(w_{(1)})\o  \z^{-1}_{W}(w_{(0)}),
\end{align*}
for any $v\in V$ and $w\in W$.
\end{theorem}
We can refer to the proof of Theorem 3.6 in \cite{ZGW19}.
\section{A category equivalence}
\def\theequation{4.\arabic{equation}}
\setcounter{equation} {0}
Let $(H,\b)$ be a Hom-Hopf algebra such that $\b$ is bijective, in this section, we will give the equivalence between the monoidal category $(~\!^{H}_{H}\mathfrak{M}^{H}_{H},\o_{H})$ or $(~\!^{H}_{H}\mathfrak{M}^{H}_{H},\B_{H})$, and the new monoidal category $\mathcal{YD}^{H}_{H}$ over $H$, which generalizes the main result in \cite{S94}.

\begin{lemma}\label{L2}
Let $V$ be a linear space and $\z_{V}:V\mapsto V$ a bijection. Endow the object $(H\o V, \b\o \z_{V})\in \!^{H}_{H}\mathfrak{M}$ with the structures given as in Example $\ref{E1}$ $(1)$. Then there is a bijection between
\begin{enumerate}
\item[$(1)$] a right $H$-comodule structures on $H\o V$ making $(H\o V, \b\o \z_{V})$ an object of $\!^{H}_{H}\mathfrak{M}^{H}$;
\item[$(2)$] a right $H$-comodule structures on $ V$ making $\e_{H}\o V:V\ra H\o V$ a morphism of right $H$-comodules.
\end{enumerate}
\end{lemma}
\begin{proof} $(2)\Rightarrow (1)$
If $V$ is a right $H$-comodule, for any $h \in H$ and $v\in V$, we define
\begin{align}\label{e4.6}
\r^{r}:H\o V\ra H\o V\o H, \r^{r}(h\o v)=h_{1}\o v_{(0)}\o \b^{-1}(h_{2}v_{(1)}),
\end{align}
for any $h,g\in H$ and $v\in V$.
By the definition of two-cosided $H$-Hopf module, we only need to prove that $(H\o V, \b\o \z_{V})$ is a right $H$-comodule and the object $(H\o V, \b\o \z_{V})\in \!_{H}\mathfrak{M}^{H}$. We first prove that $(H\o V, \b\o \z_{V})$ is a right $H$-comodule. For any $h \in H$ and $v\in V$, we have
\begin{align*}
(\r^{r}\o \b)\r^{r}(h\o v)&=(\r^{r}\o \b)(h_{1}\o v_{(0)}\o \b^{-1}(h_{2}v_{(1)}))
\\&=h_{11}\o v_{(0)(0)}\o \b^{-1}(h_{12}v_{(0)(1)})\o h_{2}v_{(1)}
\\&=\b(h_{1})\o \z_{V}(v_{(0)})\o \b^{-1}(h_{21}v_{(1)1})\o \b^{-1}(h_{22}v_{(1)2})
\\&=(\b\o \z_{V}\o \D)(h_{1}\o v_{(0)}\o \b^{-1}(h_{2}v_{(1)}))
\\&=(\b\o \z_{V}\o \D)\r^{r}(h\o v)
\end{align*}
It is easy to check that the equation $(H\o V\o \v)\r^{r}(h\o v)=\b(h)\o \z_{V}(v)$ holds.

Finally, we verify that the compatibility condition $(\ref{e2.2})$. For any $h,g \in H$ and $v\in V$, we have
\begin{align*}
\r^{r}(h\c (g\o v))&=\r^{r}(hg\o \z_{V}(v))
\\&=h_{1}g_{1}\o \z_{V}(v_{(0)})\o \b^{-1}(h_{2}g_{2})v_{(1)}
\\&=h_{1}\c(g_{1}\o v_{(0)})\o h_{2}\b^{-1}(g_{2}v_{(1)})
\\&=h_{1}\c(g\o v)_{(0)}\o h_{2}(g\o v)_{(1)}.
\end{align*}
$(1)\Rightarrow (2)$ If $(H\o V, \b\o \z_{V})\in \!^{H}_{H}\mathfrak{M}^{H}$ with the right $H$-comodule structure $\r^{r}:H\o V\ra H\o V\o H$. Then there is a unique right $H$-comodule structure on $V$ given by $\r=(\v\o V\o H)\r^{r}(\e\o V):V\ra V\o H$.

First, applying $(\e_{H}\o V\o H)$ to both sides of the equation above, we obtain
\begin{align*}
(\e_{H}\o V\o H)\r=(\e_{H}\o V\o H)(\v_{H}\o V\o H)\r^{r}(\e_{H}\o V)=\r^{r}(\e_{H}\o V).
\end{align*}
Thus $\e_{H}\o V$ is a morphism of right $H$-comodules.

Next, we prove that $V$ is a right $H$-comodule. For any $h,g\in H$, $m\in M$ and $n\in N$, we have
\begin{align*}
(\r\o \b)\r=&(\v_{H}\o V\o H\o H)(\r^{r}\o \b)(\e_{H}\o V\o H)(\v_{H}\o V\o H)\r^{r}(\e_{H}\o V)
\\&=(\v_{H}\o V\o H\o H)(\r^{r}\o \b)\r^{r}(\e_{H}\o V)
\\&=(\v_{H}\o V\o H\o H)(\b\o\z_{V}\o \D)\r^{r}(\e_{H}\o V)
\\&=(\z_{V}\o \D)(\v_{H}\o V\o H)\r^{r}(\e_{H}\o V)
\\&=(\z_{V}\o \D)\r
\end{align*}
and
\begin{align*}
(V\o \v_{H})\r&=(V\o \v_{H})(\v_{H}\o V\o H)\r^{r}(\e_{H}\o V)
\\&=(\v_{H}\o V)(H\o V\o \v_{H})\r^{r}(\e_{H}\o V)
\\&=(\v_{H}\o V)(\b\o \z_{V})(\e_{H}\o V)
\\&=\z_{V}.
\end{align*}

This completes the proof.
\end{proof}
The proof of Lemma \ref{L2} was given in the framework of general monoidal categories. Applying the lemma to the opposite category gives:
\begin{lemma}\label{L1}
Let $V$ be a linear space and $\z_{V}:V\mapsto V$ a bijection. Endow $(H\o V, \b\o \z_{V})\in \!^{H}_{H}\mathfrak{M}$ with the structures given as in Example $\ref{E1}$ $(1)$. Then there is a bijection between
\begin{enumerate}
\item[$(1)$] right $H$-module structures on $H\o V$ making $(H\o V, \b\o \z_{V})$ an object of $\!^{H}_{H}\mathfrak{M}_{H}$;
\item[$(2)$] right $H$-module structures on $ V$ making $\v\o V:H\o V\ra V$ a morphism of right $H$-modules.
\end{enumerate}
\end{lemma}
If $V$ is a right $H$-module. The induced right $H$-module structure on $H\o V$ are defined as follows, for any $h,g\in H$ and $v\in V$:
\begin{align}
(g\o v)\c h&=g\b^{-1}(h_{1})\o v\tl \b^{-1}(h_{2}). \label{e4.5}
\end{align}
The proof is analogous to that of Lemma \ref{L2} and is omitted for brevity.
\begin{theorem}
Let $V$ be a linear space and $\z_{V}:V\mapsto V$ a bijection. Endow the object $(H\o V, \b\o \z_{V})\in \!^{H}_{H}\mathfrak{M}$ with the structures given as in Example $\ref{E1}$ $(1)$. Then there is a bijection between
\begin{enumerate}
\item[$(1)$] a right $H$-module structure and a right $H$-comodule structure on $H\o V$
making $(H\o V,\b\o\z_{V})$ an object of $\!^{H}_{H}\mathfrak{M}^{H}_{H}$;
\item[$(2)$] a structure of right-right Yetter-Drinfel'd module on $V$.
\end{enumerate}
\end{theorem}
\begin{proof}
Based on Lemma \ref{L2} and \ref{L1}, we only need to prove that the condition on the right $H$-module structure and right $H$-comodule structure on $(V,\z_{V})$ that they define a right-right Yetter-Drinfel'd module is equivalent to the condition making $(H\o V,\b\o\z_{V})$ an object of $\mathfrak{M}^{H}_{H}$.

Let $(V,\z_{V},\tl)$ be a right $H$-module and $(V,\z_{V}, \r)$ a right $H$-comodule. The induced right $H$-module structure and right $H$-comodule structure on $H\o V$ are defined in (\ref{e4.5}) and (\ref{e4.6}), respectively.

Then for any $h,g\in H$ and $v\in V$, we have
\begin{align*}
\r^{r}((g\o v)\c h)&=\r^{r}(g\b^{-1}(h_{1})\o v\tl \b^{-1}(h_{2}))
\\&=g_{1}\b^{-1}(h_{11})\o (v\tl \b^{-1}(h_{2}))_{(0)}\o \b^{-1}[g_{2}\b^{-1}(h_{12})]\b^{-1} ((v\tl \b^{-1}(h_{2}))_{(1)})
\\&=g_{1}h_{1}\o (v\tl \b^{-2}(h_{22}))_{(0)}\o [\b^{-1}(g_{2})\b^{-2}(h_{21})]\b^{-1} ((v\tl \b^{-2}(h_{22}))_{(1)})
\\&=g_{1}h_{1}\o (v\tl \b^{-2}(h_{22}))_{(0)}\o g_{2}\b^{-2}[h_{21}(v\tl \b^{-2}(h_{22}))_{(1)}]
\end{align*}
and
\begin{align*}
(g\o v)_{(0)}&\c h_{1}\o (g\o v)_{(1)}h_{2}
\\&=(g_{1}\o v_{(0)})\c h_{1}\o \b^{-1} (g_{2}v_{(1)})h_{2}
\\&=g_{1}\b^{-1}(h_{11})\o v_{(0)}\tl \b^{-1}(h_{12})\o g_{2}[\b^{-1} (v_{(1)})\b^{-1}(h_{2})]
\\&=g_{1}h_{1}\o v_{(0)}\tl \b^{-1}(h_{21})\o g_{2}\b^{-2}[\b (v_{(1)})h_{22}],
\end{align*}
we easily see that these two terms are equal if $(V,\z_{V})$ is a right-right Yetter-Drinfel'd module. Thus $(H\o V,\b\o \z_{V})$ is a right-right $H$-Hopf module over $H$.

Conversely, assuming that $(H\o V,\b\o \z_{V})$ is an object of $\mathfrak{M}^{H}_{H}$, that is, for any $h\in H$ and $v\in V$, we have
\begin{align*}
\r^{r}((1\o v)\c h)=(1\o v)_{(0)}\c h_{1}\o (1\o v)_{(1)}h_{2},
\end{align*}
applying $\v\o V\o \b $ to both sides of the equation above, we get
\begin{align*}
v_{(0)}\tl &h_{1}\o \b (v_{(1)})\b (h_{2})
\\&=(\v\o V\o \b )(h_{11}\o v_{(0)}\tl \b^{-1}(h_{12})\o v_{(1)} h_{2})
\\&=(\v\o V\o \b )((1\o v)_{(0)}\c h_{1}\o (1\o v)_{(1)}h_{2})
\\&=(\v\o V\o \b )(\r^{r}((1\o v)\c h))
\\&=(\v\o V\o \b )(h_{11}\o (v\tl \b^{-1}(h_{2}))_{(0)}\o \b^{-1}(h_{12}(v\tl \b^{-1}(h_{2}))_{(1)}))
\\&=(v\tl \b^{-1}(h_{2}))_{(0)}\o \b(h_{1})(v\tl \b^{-1}(h_{2}))_{(1)},
\end{align*}
replacing $h$ by $\b(h)$, we obtain Eq.\ref{E1}. Thus $(V,\z_{V})$ is a right-right Yetter-Drinfel'd module over $H$.

This completes the proof.
\end{proof}
\begin{theorem}\label{T1}
Let $(H,\b)$ be a Hom-Hopf algebra with $\b$ bijective. Then the equivalence
\begin{align*}
\!^{H}_{H}\mathfrak{M}&\cong \mathfrak{C}
\\H\o V&\leftarrow V
\\ M&\ra \!^{coH}M
\end{align*}
induces equivalences of monoidal categories between
\begin{enumerate}
\item[$(1)$] the category $\!^{H}_{H}\mathfrak{M}_{H}$ of two-sided $H$-Hopf modules with Hom-tensor product $\o_{H}$ and the category of right $H$-modules,
\item[$(2)$] the category $\!^{H}_{H}\mathfrak{M}^{H}$ of two-cosided $H$-Hopf modules with Hom-cotensor product $\B_{H}$ and the category of right $H$-comodules,
\item[$(3)$] the category $\!^{H}_{H}\mathfrak{M}^{H}_{H}$ of four-angle Hopf modules with either $\o_{H}$ or $\B_{H}$ as product structure, and the category of right-right Yetter-Drinfel'd modules over $H$,
\end{enumerate}
where the right (co)module structures on $H\o V$ for $V$ a right (co)module is Eq.$(\ref{e4.5})$ and $(\ref{e4.6})$.
\end{theorem}
\begin{proof}We define the subspace of $M$ by
\begin{align*}
\!^{coH}M=\{m\in M~|~\r^{l}(m)=1\o \z_{M}(m)\}.
\end{align*}
The right $H$-comodule structure on $\!^{coH}M$ for $M\in \!^{H}_{H}\mathfrak{M}^{H}$ is that of $\!^{coH}M$ as a right $H$-subcomodule of $M$. The right $H$-module structure on $\!^{coH}M$ for $M\in \!^{H}_{H}\mathfrak{M}_{H}$ is defined by $m'\tl h=S\b^{-1}(h_{1})\c(\z_{M}^{-1}(m')\c \b^{-2}(h_{2}))$, for any $h\in H$ and $m'\in \!^{coH}M$. We first check that the action is well defined. For any $h\in H$ and $m'\in \!^{coH}M$, we have
\begin{align*}
\r^{l}(m'\tl h)&=\r^{l}(\b^{-1}S(h_{1})\c (\z_{M}^{-1}(m')\c \b^{-2}(h_{2})))
\\&=\b^{-1 }S(h_{12})(\z_{M}^{-1}(m')\c \b^{-2}(h_{2}))_{[-1]}\o \b^{-1}S(h_{11})\c (\z_{M}^{-1}(m')\c \b^{-2}(h_{2}))_{[0]}
\\&=\b^{-1 }S(h_{12})(\z_{M}^{-1}(m')_{[-1]}\b^{ -2}(h_{21}))\o \b^{-1}S(h_{11})\c (\z_{M}^{-1}(m')_{[0]}\c \b^{-2}(h_{22}))
\\&=\b^{-1 }S(h_{12})\b^{ -1}(h_{21})\o \b^{-1}S(h_{11})\c (m'\c \b^{-2}(h_{22}))
\\&=\b^{-2}S(h_{211})\b^{ -2}(h_{212})\o  S(h_{1})\c (m'\c \b^{-2}(h_{22}))
\\&=1\o  S(h_{1})\c (m'\c \b^{-1}(h_{2}))
\\&=1\o \z_{M}(m'\tl h).
\end{align*}
It is easy to show that $(\!^{coH}M,\z_{M})\in \mathfrak{M}_{H}$.

Next, we only need to check the assertion that we have monoidal equivalences. To do this, it is enough to prove that one of the quasi-inverse equivalences is a monoidal functor in each case.

For $(1)$ we show that the isomorphism
\begin{align*}
\vp:(H\o V)\o_{H}(H\o W) &\ra H\o V\o W
\\ g\o v\o_{H} h\o w &\m(\b^{-1}(g)\o \z_{V}^{-1}(v))\c \b^{-1}(h)\o w
\\ h\o v\o_{H} 1\o w&\mapsfrom h\o v\o w
\end{align*}
is a morphism in the category $\!^{H}_{H}\mathfrak{M}_{H}^{H}$.
For left linearity and colinearity, computing we have
\begin{align*}
\vp[k\c (g\o v\o_{H} h\o w)]&=\vp[k\c (g\o v)\o_{H} \b(h)\o \z_{W}(w)]
\\&=\vp[kg\o \z_{V}(v)\o_{H} \b(h)\o \z_{W}(w)]
\\&=\b^{ -1}(kg)\b^{-1}(h_{1})\o v\tl \b^{-1}(h_{2})\o \z_{W}(w)
\\&=k[\b^{-1}(g)\b^{-2}(h_{1})]\o v\tl \b^{-1}(h_{2})\o \z_{W}(w)
\\&=k\c (\b^{-1}(g)\b^{-2}(h_{1})\o \z_{V}^{-1}(v)\tl \b^{-2}(h_{2})\o w)
\\&=k\c (\vp[g\o v\o_{H} h\o w])
\end{align*}
and
\begin{align*}
[\vp(g\o v&\o_{H} h\o w)]_{[-1]}\o [\vp(g\o v\o_{H} h\o w)]_{[0]}
\\&=[\b^{-1}(g)\b^{-2}(h_{1})\o \z_{V}^{-1}(v)\tl \b^{-2}(h_{2})\o w]_{[-1]}
\\&\quad\o [\b^{-1}(g)\b^{-2}(h_{1})\o \z_{V}^{-1}(v)\tl \b^{-2}(h_{2})\o w]_{[0]}
\\&=\b^{-1}(g_{1})\b^{ -2}(h_{11})\o \b^{-1}(g_{2})\b^{-2}(h_{12})\o v\tl \b^{-1}(h_{2})\o \z_{W}(w)
\\&=\b^{-1}(g_{1})\b^{ -1}(h_{1})\o \b^{-1}(g_{2})\b^{-2}(h_{21})\o v\tl \b^{-2}(h_{22})\o \z_{W}(w)
\\&=\b^{-1}(g_{1}h_{1})\o \vp[g_{2}\o \z_{V}(v)\o_{H} h_{2}\o \z_{W}(w)]
\\&=\b^{-1}((g\o v)_{[-1]}(h\o w)_{[-1]})\o \vp[(g\o v)_{[0]}\o_{H} (h\o w)_{[0]}]
\\&=(g\o v\o h\o w)_{[-1]} \o \vp[(g\o v\o_{H} h\o w)_{[0]}],
\end{align*}
for any $g,h\in H$, $v\in V$ and $w\in W$.
For right linearity, we have
\begin{align*}
\vp[(g\o v&\o_{H} h\o w)\c k]
\\&=\vp[\b(g)\o\z_{V}( v)\o_{H} (h\o w)\c k]
\\&=\vp[\b(g)\o\z_{V}( v)\o_{H} h\b^{-1}(k_{1})\o w\tl \b^{-1}(k_{2})]
\\&=g[\b^{-2}(h_{1})\b^{-3}(k_{11})]\o v\tl \b^{-2}(h_{2})\b^{-3}(k_{12})\o w \tl \b^{-1}(k_{2})
\\&=[\b^{-1}(g)\b^{-2}(h_{1})]\b^{-2}(k_{11})\o [\z_{V}^{-1}( v)\tl \b^{-2}(h_{2})]\tl \b^{-2}(k_{12}) \o w \tl \b^{-1}(k_{2})
\\&=[\b^{-1}(g)\b^{-2}(h_{1})]\b^{-1}(k_{1})\o [\z_{V}^{-1}( v)\tl \b^{-2}(h_{2})]\tl \b^{-2}(k_{21}) \o w \tl \b^{-2}(k_{22})
\\&=[\b^{-1}(g)\b^{-2}(h_{1})]\b^{-1}(k_{1})\o [\z_{V}^{-1}( v)\tl \b^{-2}(h_{2})\o w]\tl \b^{-1}(k_{2})
\\&=[\b^{-1}(g)\b^{-2}(h_{1})\o \z_{V}^{-1}( v)\tl \b^{-2}(h_{2})\o w]\c k
\\&=[\vp(g\o v\o_{H} h\o w)]\c k.
\end{align*}
Part $(2)$ is formally dual to $(1)$. We only deal with the half of $(3)$ involving $\o_{H}$ since the other half is dual to this. It remains to check that $\vp$ is right colinearity. For any $g,h\in H$, $v\in V$ and $w\in W$, we get
\begin{align*}
[\vp(g\o v&\o_{H} h\o w)]_{(0)}\o [\vp(g\o v\o_{H} h\o w)]_{(1)}
\\&=[\b^{-1}(g )\b^{-2}(h_{1})\o \z_{V}^{-1}(v )\tl \b^{-2}(h_{2})\o w]_{(0)}
\\&\quad\o [\b^{-1}(g )\b^{-2}(h_{1})\o \z_{V}^{-1}(v )\tl \b^{-2}(h_{2})\o w]_{(1)}
\\&=\b^{-1}(g_{1})\b^{-2}(h_{11})\o [\z_{V}^{-1}(v )\tl \b^{-2}(h_{2})\o w]_{(0)}
\\&\quad\o \b^{-1}[[\b^{-1}(g_{2})\b^{-2}(h_{12})][\z_{V}^{-1}(v )\tl \b^{-2}(h_{2})\o w]_{(1)}]
\\&=\b^{-1}(g_{1})\b^{-2}(h_{11})\o [\z_{V}^{-1}(v )\tl \b^{-2}(h_{2})]_{(0)}\o w_{(0)}
\\&\quad\o \b^{-1}[[\b^{-1}(g_{2})\b^{-2}(h_{12})]\b^{-1}[(\z_{V}^{-1}(v )\tl \b^{-2}(h_{2})_{(1)})w_{(1)}]]
\\&=\b^{-1}(g_{1})\b^{-2}(h_{11})\o [\z_{V}^{-1}(v )\tl \b^{-2}(h_{2})]_{(0)}\o w_{(0)}
\\&\quad\o [\b^{-2}(g_{2})\b^{-3}(h_{12})][\b^{-2}(\z_{V}^{-1}(v )\tl \b^{-2}(h_{2})_{(1)})\b^{-2}(w_{(1)})]
\\&=\b^{-1}(g_{1})\b^{-2}(h_{11})\o [\z_{V}^{-1}(v )\tl \b^{-2}(h_{2})]_{(0)}\o w_{(0)}
\\&\quad\o \b^{-1}(g_{2})[[\b^{-4}(h_{12})\b^{-3}(\z_{V}^{-1}(v )\tl \b^{-2}(h_{2})_{(1)})]\b^{-2}(w_{(1)})]
\\&=\b^{-1}(g_{1})\b^{-2}(h_{11})\o [\z_{V}^{-1}(v )\tl \b^{-2}(h_{2})]_{(0)}\o w_{(0)}
\\&\quad\o \b^{-1}(g_{2})[\b^{-3}[\b^{-1}(h_{12})(\z_{V}^{-1}(v )\tl \b^{-2}(h_{2})_{(1)})]\b^{-2}(w_{(1)})]
\\&=\b^{-1}(g_{1})\b^{-1}(h_{1})\o [\z_{V}^{-1}(v )\tl \b^{-3}(h_{22})]_{(0)}\o w_{(0)}
\\&\quad\o \b^{-1}(g_{2})[\b^{-3}[\b^{-1}(h_{21})(\z_{V}^{-1}(v )\tl \b^{-3}(h_{22})_{(1)})]\b^{-2}(w_{(1)})]
\\&=\b^{-1}(g_{1})\b^{-1}(h_{1})\o [\z_{V}^{-1}(v_{(0)} )\tl \b^{-2}(h_{21})]\o w_{(0)}
\\&\quad\o \b^{-1}(g_{2})[\b^{-3}[v_{(1)}\b^{-1}(h_{22})]\b^{-2}(w_{(1)})]
\\&=\b^{-1}(g_{1})\b^{-2}(h_{11})\o \z_{V}^{-1}(v_{(0)})\tl \b^{-2}(h_{12})\o w_{(0)}
\\&\quad\o  \b^{-1}(g_{2})[[\b^{-3}(v_{(1)})\b^{-3}(h_{2})]\b^{-2}(w_{(1)})]
\\&=\b^{-1}(g_{1})\b^{-2}(h_{11})\o \z_{V}^{-1}(v_{(0)})\tl \b^{-2}(h_{12})\o w_{(0)}
\\&\quad\o  \b^{-1}(g_{2})[[\b^{-3}(v_{(1)})\b^{-3}(h_{2})]\b^{-2}(w_{(1)})]
\\&=\b^{-1}(g_{1})\b^{-2}(h_{11})\o \z_{V}^{-1}(v_{(0)})\tl \b^{-2}(h_{12})\o w_{(0)}
\o  \b^{-2}(g_{2}v_{(1)})\b^{-2}(h_{2}w_{(1)})
\\&=\vp[g_{1}\o v_{(0)}\o_{H} h_{1}\o w_{(0)}]\o \b^{-1}[\b^{-1}(g_{2}v_{(1)})\b^{-1}(h_{2}w_{(1)})]
\\&=\vp[(g\o v)_{(0)}\o_{H} (h\o w)_{(0)}]\o \b^{-1}[(g\o v)_{(1)}(h\o w)_{(1)}]
\\&=\vp[(g\o v\o_{H} h\o w)_{(0)}]\o (g\o v\o_{H} h\o w)_{(1)}.
\end{align*}
\begin{align*}
[\vp(g\o v&\o_{H} h\o w)]_{(0)}\o [\vp(g\o v\o_{H} h\o w)]_{(1)}
\\&=[\b^{-1}(g )\b^{-2}(h_{1})\o \z_{V}^{-1}(v )\tl \b^{-2}(h_{2})\o w]_{(0)}
\\&\quad\o [\b^{-1}(g )\b^{-2}(h_{1})\o \z_{V}^{-1}(v )\tl \b^{-2}(h_{2})\o w]_{(1)}
\\&=\b^{-1}(g_{1})\b^{-2}(h_{11})\o [\z_{V}^{-1}(v )\tl \b^{-2}(h_{2})\o w]_{(0)}
\\&\quad\o \b^{-1}[[\b^{-1}(g_{2})\b^{-2}(h_{12})][\z_{V}^{-1}(v )\tl \b^{-2}(h_{2})\o w]_{(1)}]
\\&=\b^{-1}(g_{1})\b^{-2}(h_{11})\o [\z_{V}^{-1}(v )\tl \b^{-2}(h_{2})]_{(0)}\o w_{(0)}
\\&\quad\o \b^{-1}[[\b^{-1}(g_{2})\b^{-2}(h_{12})]\b^{-1}[(\z_{V}^{-1}(v )\tl \b^{-2}(h_{2})_{(1)})w_{(1)}]]
\\&=\b^{-1}(g_{1})\b^{-2}(h_{11})\o [\z_{V}^{-1}(v )\tl \b^{-2}(h_{2})]_{(0)}\o w_{(0)}
\\&\quad\o [\b^{-2}(g_{2})\b^{-3}(h_{12})][\b^{-2}(\z_{V}^{-1}(v )\tl \b^{-2}(h_{2})_{(1)})\b^{-2}(w_{(1)})]
\\&=\b^{-1}(g_{1})\b^{-2}(h_{11})\o [\z_{V}^{-1}(v )\tl \b^{-2}(h_{2})]_{(0)}\o w_{(0)}
\\&\quad\o \b^{-1}(g_{2})[[\b^{-4}(h_{12})\b^{-3}(\z_{V}^{-1}(v )\tl \b^{-2}(h_{2})_{(1)})]\b^{-2}(w_{(1)})]
\\&=\b^{-1}(g_{1})\b^{-2}(h_{11})\o [\z_{V}^{-1}(v )\tl \b^{-2}(h_{2})]_{(0)}\o w_{(0)}
\\&\quad\o \b^{-1}(g_{2})[\b^{-3}[\b^{-1}(h_{12})(\z_{V}^{-1}(v )\tl \b^{-2}(h_{2})_{(1)})]\b^{-2}(w_{(1)})]
\\&=\b^{-1}(g_{1})\b^{-1}(h_{1})\o [\z_{V}^{-1}(v )\tl \b^{-3}(h_{22})]_{(0)}\o w_{(0)}
\\&\quad\o \b^{-1}(g_{2})[\b^{-3}[\b^{-1}(h_{21})(\z_{V}^{-1}(v )\tl \b^{-3}(h_{22})_{(1)})]\b^{-2}(w_{(1)})]
\\&=\b^{-1}(g_{1})\b^{-1}(h_{1})\o [\z_{V}^{-1}(v_{(0)} )\tl \b^{-2}(h_{21})]\o w_{(0)}
\\&\quad\o \b^{-1}(g_{2})[\b^{-3}[v_{(1)}\b^{-1}(h_{22})]\b^{-2}(w_{(1)})]
\\&=\b^{-1}(g_{1})\b^{-2}(h_{11})\o \z_{V}^{-1}(v_{(0)})\tl \b^{-2}(h_{12})\o w_{(0)}
\\&\quad\o  \b^{-1}(g_{2})[[\b^{-3}(v_{(1)})\b^{-3}(h_{2})]\b^{-2}(w_{(1)})]
\\&=\b^{-1}(g_{1})\b^{-2}(h_{11})\o \z_{V}^{-1}(v_{(0)})\tl \b^{-2}(h_{12})\o w_{(0)}
\\&\quad\o  \b^{-1}(g_{2})[[\b^{-3}(v_{(1)})\b^{-3}(h_{2})]\b^{-2}(w_{(1)})]
\\&=\b^{-1}(g_{1})\b^{-2}(h_{11})\o \z_{V}^{-1}(v_{(0)})\tl \b^{-2}(h_{12})\o w_{(0)}
\o  \b^{-2}(g_{2}v_{(1)})\b^{-2}(h_{2}w_{(1)})
\\&=\vp[g_{1}\o v_{(0)}\o_{H} h_{1}\o w_{(0)}]\o \b^{-1}[\b^{-1}(g_{2}v_{(1)})\b^{-1}(h_{2}w_{(1)})]
\\&=\vp[(g\o v)_{(0)}\o_{H} (h\o w)_{(0)}]\o \b^{-1}[(g\o v)_{(1)}(h\o w)_{(1)}]
\\&=\vp[(g\o v\o_{H} h\o w)_{(0)}]\o (g\o v\o_{H} h\o w)_{(1)}.
\end{align*}
Finally, the coherence condition on monoidal functors follows from the fact that both ways around the rectangle
$$\xymatrix{
  (H\o U)\o_{H}(H\o V)\o_{H}(H\o W) \ar[d]_{id\o_{H}\vp} \ar[r]^-{id\o_{H}\vp}      & (H\o U)\o_{H}(H\o V\o W)\ar[d]^{\vp}  \\
  (H\o U\o V)\o_{H}(H\o W)  \ar[r]_-{\vp}               & H\o U\o V\o W            }
  $$
  are given by $h\o u\o g\o v\o f\o w\ra \b^{-1}(h)(\b^{-3}(g_{1})\b^{-4}(f_{11}))\o \z_{U}^{-1}(u)\tl \b^{-3}(g_{2})\b^{-4}(f_{12})\o \z_{V}^{-1}(v)\tl \b^{-2}(f_{2})\o w$.

This completes the proof.
\end{proof}
\begin{example}
Let $H_{4}=sp\{1,g, x,gx\}$ be a vector space over $\Bbbk$ with char $\Bbbk\neq 2$ satisfying the following relation:
\begin{align*}
g^{2}=1,x^{2}=0,xg=-gx.
\end{align*}

For all $0\neq k\in \Bbbk$, define the Hom-Hopf algebra structure on $H_{4}$ as follows:
\begin{enumerate}
\item[$\bullet$] The multiplication $\circ$ is given by:
\begin{align*}
\begin{tabular}{c|c c c c}
	       H&1&g&x&gx\\
	\hline 1&1&g&kx&kgx\\
	       g&g&1&kgx&kx\\
	       x&kx&$-$kgx&0&0\\
           gx&kgx&$-$kx&0&0\\
\end{tabular}
\end{align*}
\item[$\bullet$] The comultiplication $\D$, counit $\v$ and antipode $S$ are given by:
\begin{align*}
&\D(1)=1\o 1, \quad\D(g)=g\o g, \quad \D(x)=kx\o g+1\o kx,\quad \D(gx)=kgx\o 1+g\o kgx,
\\&\v(1)=1_{\Bbbk},\quad\v(g)=1_{\Bbbk},\quad \v(x)=0, \quad\v(gx)=0,
\\&S(1)=1,\quad S(g)=g,\quad S(x)=gx, \quad S(gx)=-x.
\end{align*}
\end{enumerate}
The automorphism $\b: H_{4}\ra H_{4}$ is given by
\begin{align*}
&\b(1)=1, \quad \b(g)=g, \quad \b(x)=kx, \quad \b(gx)=kgx.
\end{align*}
Let $V=sp\{1_{V}, z\}$ over $\Bbbk$ with char $\Bbbk\neq 2$ and define the automorphism $\z_{V}: V\ra V$ by $\z_{V}(1)=1, \z_{V}(z)=kz$.
Define the action $\tl:V\o H_{4}\ra V$ by
\begin{align*}
&1_{V}\tl 1=1_{V},\quad 1_{V}\tl g=1_{V},\quad 1_{V}\tl x=0,\quad 1_{V}\tl y=0,
\\&z\tl 1=kz,\quad z\tl g=-kz,\quad z\tl x=0,\quad z\tl y=0.
\end{align*}
Define the coaction $\r: V\ra V\o H_{4}$ by
\begin{align*}
&\r(1_{V})=1_{V}\o 1,\quad \r(z)=kz\o g+1_{V}\o kx.
\end{align*}
Then $(V, \z_{V})$ is a right-right Yetter-Drinfel'd module,
where $0\neq k\in \Bbbk$. Then one can check that $(H_{4}\o V, \b\o \z_{V})$ is a four-angle Hopf module with the following structures:
\begin{enumerate}
\item[$(i)$] the left module structures:
\begin{align*}
\begin{tabular}{c|c c c c}
	\hline $\c$&$1\o 1_{V}$&$g\o 1_{V}$&$x\o 1_{V}$&$gx\o 1_{V}$\\
	\hline 1&$1\o 1_{V}$&$g\o 1_{V}$&$kx\o 1_{V}$&$kgx\o 1_{V}$\\
	       g&$g\o 1_{V}$&$1\o 1_{V}$&$kgx\o 1_{V}$&$kx\o 1_{V}$\\
	       x&$kx\o 1_{V}$&$-kgx\o 1_{V}$&0&0\\
           gx&$kgx\o 1_{V}$&$-kx\o 1_{V}$&0&0\\
    \hline $\c$&$1\o z$&$g\o z$&$x\o z$&$y\o z$\\
	\hline 1&$1\o kz$&$g\o kz$&$kx\o kz$&$kgx\o kz$\\
	       g&$g\o kz$&$1\o kz$&$kgx\o kz$&$kx\o kz$\\
	       x&$kx\o kz$&$-kgx\o kz$&0&0\\
           gx&$kgx\o kz$&$-kx\o kz$&0&0\\
\end{tabular}
\end{align*}
\item[$(ii)$] the right module structures:
\begin{align*}
\begin{tabular}{c c c c|c}
	\hline $1\o 1_{V}$&$g\o 1_{V}$&$x\o 1_{V}$&$gx\o 1_{V}$&$\c$\\
	\hline $1\o 1_{V}$&$g\o 1_{V}$&$kx\o 1_{V}$&$kgx\o 1_{V}$&1\\
	       $g\o 1_{V}$&$1\o 1_{V}$&$-kgx\o 1_{V}$&$-kx\o 1_{V}$&g\\
	       $kx\o 1_{V}$&$kgx\o 1_{V}$&0&0&x\\
           $kgx\o 1_{V}$&$kx\o 1_{V}$&0&0&gx\\
    \hline $1\o z$&$g\o z$&$x\o z$&$gx\o z$&$\c$\\
	\hline $1\o kz$&$g\o kz$&$kx\o kz$&$kgx\o kz$&1\\
	       $g\o -kz$&$1\o -kz$&$kgx\o kz$&$kx\o kz$&g\\
	       $kx\o -kz$&$kgx\o -kz$&0&0&x\\
           $kgx\o kz$&$kx\o kz$&0&0&gx\\
\end{tabular}
\end{align*}
\item[$(iii)$] the left comodule structures:
\begin{align*}
&\r^{l}(1\o 1_{V})=1\o 1\o 1_{V},\quad \r^{l}(1\o z)=1\o 1\o kz,
\\&\r^{l}(g\o 1_{V})=g\o g\o 1_{V},\quad \r^{l}(g\o z)=g\o g\o kz,
\\&\r^{l}(x\o 1_{V})=(kx\o g+1\o kx)\o 1_{V},\quad \r^{l}(x\o z)=(kx\o g+1\o kx)\o kz,
\\&\r^{l}(gx\o 1_{V})=(kgx\o 1+g\o kgx)\o 1_{V},,\quad \r^{l}(gx\o z)=(kgx\o 1+g\o kgx)\o kz.
\end{align*}
\item[$(iv)$] the right comodule structures:
\begin{align*}
&\r^{r}(1\o 1_{V})=1\o 1_{V}\o 1,\quad \r^{r}(1\o z)=1\o kz\o g+1\o 1_{V}\o kx,
\\&\r^{r}(g\o 1_{V})=g\o 1_{V}\o g,\quad \r^{r}(g\o z)=g\o kz\o 1+g\o 1_{V}\o kgx,
\\&\r^{r}(x\o 1_{V})=kx\o 1_{V}\o g+1\o 1\o kx,
\\&\r^{r}(x\o z)=kx\o kz\o 1+kx\o 1\o kgx+1\o kz\o -kgx,
\\&\r^{r}(gx\o 1_{V})=kgx\o 1_{V}\o 1+g\o 1\o kgx,
\\&\r^{r}(gx\o z)=kgx\o kz\o g+kgx\o 1\o kx+g\o kz\o -kx,
\end{align*}
\end{enumerate}
We define
$\!^{coH}(H_{4}\o V)=\{h\o v|h_{1}\o h_{2}\o \z(v)=1\o h\o \z(v)\}$. That is
\begin{align*}
&\!^{coH}(H_{4}\o V)=\{1\o 1_{V}, 1\o z\}.
\end{align*}
Then it is not hard to check that $\!^{coH}(H_{4}\o V)$ and $V$ is an isomorphism of Yetter-Drinfel'd modules.

\end{example}
\begin{corollary}
Let $(H,\b)$ be a Hom-Hopf algebra such that $\b$ is bijective. Then the identity functor is a monoidal equivalence
\begin{align*}
(~\mathfrak{ID},\xi~): (~\!^{H}_{H}\mathfrak{M}\!^{H}_{H},\B_{H})\ra (~\!^{H}_{H}\mathfrak{M}\!^{H}_{H},\o_{H}),
\end{align*}
where the isomorphisms $\xi:M\o_{H} N\ra M\B_{H} N$ satisfy $\xi (m\o n)=\z_{M}^{-2}(m_{(0)})\c \b^{-2}(n_{[-1]})\o \b^{-2}(m_{(1)})\c \z_{N}^{-2}(n_{[0]})$, for any $m\in M$ and $n\in N$.
\end{corollary}
\begin{proof}
The identity $\mathfrak{ID}$ is isomorphic to the composition
$$
\xymatrix@C=0.5cm{
 (~\!^{H}_{H}\mathfrak{M}^{H}_{H},\B_{H}) \ar[rr]^-{\!^{coH}(-)} &&  (\mathcal{YD}^{H}_{H},\o) \ar[rr]^-{H\o (-)} &&  (~\!^{H}_{H}\mathfrak{M}^{H}_{H},\o_{H})  }
  $$
of two monoidal equivalences. We only need to prove that the induced structure of monoidal functor on the identity has the form one have claimed. It is sufficient to consider the case $M=H\o V$ and $N=H\o W$ with $V, W\in \mathcal{YD}^{H}_{H}$. Then $\xi$ is the composition
$$
\xymatrix@C=0.5cm{
(H\o V)\o_{H} (H\o W)\ar[r]^-{\vp} &  (H\o V\o W) \ar[r]^-{\d^{-1}} & (H\o V)\B_{H} (H\o W)  },
  $$
where $\d^{-1}$ is dual to $\vp$ and is defined by $\d^{-1}(g\o v\o w)=(\b^{-1}(g)\o \z_{V}^{-1}(v))_{(0)}\o (\b^{-1}(g)\o \z_{V}^{-1}(v))_{(1)}\o w$. Thus we have
\begin{align*}
\d^{-1}\vp&(h\o v\o_{H}g\o w)=\d^{-1}[(\b^{-1}(g)\o \z^{-1}(v))\c \b^{-1}(h)\o w]
\\&=\d^{-1}[\b^{-1}(g)\b^{-2}(h_{1})\o \z^{-1}(v)\tl \b^{-2}(h_{2})\o w]
\\&=\b^{-2}(g_{1})\b^{-3}(h_{11})\o \z^{-1}[\z^{-1}(v)\tl \b^{-2}(h_{2})]_{(0)}
\\&\quad\o \b^{-2}([\b^{-1}(g_{2})\b^{-2}(h_{12})][\z^{-1}(v)\tl \b^{-2}(h_{2})]_{(1)})\o w
\\&=\b^{-2}(g_{1})\b^{-3}(h_{11})\o \z^{-1}[\z^{-1}(v)\tl \b^{-2}(h_{2})]_{(0)}
\\&\quad\o \b^{-2}(g_{2})\b^{-3}[\b^{-1}(h_{12})[\z^{-1}(v)\tl \b^{-2}(h_{2})]_{(1)}]\o w
\\&=\b^{-2}(g_{1})\b^{-2}(h_{1})\o \z^{-1}[\z^{-1}(v)\tl \b^{-3}(h_{22})]_{(0)}
\\&\quad\o \b^{-2}(g_{2})\b^{-3}[\b^{-1}(h_{21})[\z^{-1}(v)\tl \b^{-3}(h_{22})]_{(1)}]\o w
\\&=\b^{-2}(g_{1})\b^{-2}(h_{1})\o \z^{-1}[\z^{-1}(v_{(0)})\tl \b^{-2}(h_{21})]\o \b^{-2}(g_{2})\b^{-3}[v_{(1)} \b^{-1}(h_{22})]\o w
\\&=\b^{-2}(g_{1})\b^{-3}(h_{11})\o [\z^{-2}(v_{(0)})\tl \b^{-3}(h_{12})]\o \b^{-2}(g_{2})\b^{-3}[v_{(1)} h_{2}]\o w
\\&=[\b^{-2}(g_{1})\o \z^{-2}(v_{(0)})]\c \b^{-2}(h_{1})\o \b^{-3}(g_{2})\b^{-3}(v_{(1)}) \b^{-2}(h_{2})\o w
\\&=[\b^{-2}(g_{1})\o \z^{-2}(v_{(0)})]\c \b^{-2}(h_{1})\o \b^{-3}(g_{2})\b^{-3}(v_{(1)}) \c[\b^{-2}(h_{2})\o \z^{-1}(w)]
\\&=[\b^{-2}(g)\o \z^{-2}(v)]_{(0)}\c [\b^{-2}(h)\o \z^{-2}(w)]_{[-1]}
\\&\quad\o [\b^{-2}(g)\o \z^{-2}(v)]_{(1)} \c[\b^{-2}(h)\o \z^{-2}(w)]_{[0]}
\\&=\z_{M}^{-2}(m_{(0)})\c \b^{-2}(n_{[-1]})\o \b^{-2}(m_{(1)})\c \z_{N}^{-2}(n_{[0]}).
\end{align*}
The coherence condition on $\xi$ is commutativity of the diagram
$$\xymatrix{
  M\o_{H} N\o_{H} P \ar[d]_{M\o \xi} \ar[r]^{\xi\o P}  & (M\B_{H} N)\o_{H} P \ar[d]^{\xi}  \\
  M\o_{H} (N\B_{H} P)  \ar[r]_{\xi}   & M\B_{H} N\B_{H} P.             }
                $$

This completes the proof.
\end{proof}
In the following, we will construct a braiding structure on the monoidal category $(~\!^{H}_{H}\mathfrak{M}^{H}_{H},\o_{H})$, which is an important result in this section. In the same manner, we can construct a braiding structure on the monoidal category $(~\!^{H}_{H}\mathfrak{M}^{H}_{H},\B_{H})$.
\begin{theorem}\label{T2}
Let $(H,\b)$ be a Hom-Hopf algebra such that $\b$ and the antipode $S$ are bijective. Then $(~\!^{H}_{H}\mathfrak{M}^{H}_{H},\o_{H})$ is a braided monoidal category with a braiding
\begin{align*}
\widetilde{\si}:M\o_{H} N&\ra N\o_{H} M
\\m\o n&\m [\b^{-4}(m_{[-1]1})\c \z_{N}^{-3}(n_{(0)})]\c S\b^{-3}(n_{(1)1})
\\&\qquad\quad \o S\b^{-3}(m_{[-1]2})\c[\z_{M}^{-3}(m_{[0]})\c \b^{-4}(n_{(1)2})],
\end{align*}
and the inverse
\begin{align*}
\widetilde{\si}^{-1} (n\o m)=[\b^{-4}(n_{(1)2})\c &\z^{-3}_{M}(m_{[0]})]\c S^{-1}\b^{-3}(m_{[-1]2})
\\&\o S^{-1}\b^{-3}(n_{(1)1})\c [\z^{-3}_{N}(n_{(0)})\c \b^{-4}(m_{[-1]1})],
\end{align*}
for any $m\in M$ and $n\in N$.
\end{theorem}
\begin{proof}
By Theorem \ref{T1}, it remains to check that the braiding $\widetilde{\si}$ induced in $\!^{H}_{H}\mathfrak{M}^{H}_{H}$ via the monoidal equivalent with $\mathcal{YD}^{H}_{H}$ has the stated form.

We first check that the linear map $\widetilde{\si}$ is well defined. For any $h\in H$, $m\in M$ and $n\in N$, we have
\begin{align*}
\widetilde{\si}(m\c h\o_{H} \z_{N}(n))&=[\b^{-4}((m\c h)_{[-1]1})\c \z_{N}^{-2}(n_{(0)})]\c S\b^{-2}(n_{(1)1})
\\&\quad \o_{H} S\b^{-3}((m\c h)_{[-1]2})\c[\z_{M}^{-3}((m\c h)_{[0]})\c \b^{-3}(n_{(1)2})]
\\&=[\b^{-4}(m_{[-1]1} h_{11})\c \z_{N}^{-2}(n_{(0)})]\c S\b^{-2}(n_{(1)1})
\\&\quad \o_{H} S\b^{-3}(m_{[-1]2} h_{12})\c[\z_{M}^{-3}((m_{[0]}\c h_{2}))\c \b^{-3}(n_{(1)2})]
\\&=[\b^{-4}(m_{[-1]1} h_{11})\c \z_{N}^{-2}(n_{(0)})]\c S\b^{-2}(n_{(1)1})
\\&\quad\o_{H} S\b^{-2}(h_{12})\c [S\b^{-3}(m_{[-1]2})\c[\z_{M}^{-4}(m_{[0]}\c h_{2})\c \b^{-4}(n_{(1)2})]]
\\&=[[\b^{-5}(m_{[-1]1} h_{11})\c \z_{N}^{-3}(n_{(0)})]\c S\b^{-3}(n_{(1)1})]\c S\b^{-2}(h_{12})
\\&\quad\o_{H} [S\b^{-2}(m_{[-1]2})\c[\z_{M}^{-3}(m_{[0]}\c h_{2})\c \b^{-3}(n_{(1)2})]]
\\&=[\b^{-4}(m_{[-1]1} h_{11})\c \z_{N}^{-2}(n_{(0)})]\c S\b^{-3}(n_{(1)1})S\b^{-3}(h_{12})
\\&\quad\o_{H} [S\b^{-2}(m_{[-1]2})\c[\z_{M}^{-3}(m_{[0]}\c h_{2})\c \b^{-3}(n_{(1)2})]]
\\&=[\b^{-3}(m_{[-1]1})\c [\b^{-4}(h_{11})\c \z_{N}^{-3}(n_{(0)})]]\c S\b^{-3}(h_{12}n_{(1)1})
\\&\quad\o_{H} S\b^{-2}(m_{[-1]2})\c[\z_{M}^{-2}(m_{[0]})\c \b^{-3}(h_{2})\b^{-4}(n_{(1)2})]
\\&=[\b^{-3}(m_{[-1]1})\c \z_{N}^{-3}(h_{1}\c n_{(0)})]\c S\b^{-3}(h_{21}n_{(1)1})
\\&\quad\o_{H} S\b^{-2}(m_{[-1]2})\c[\z_{M}^{-2}(m_{[0]})\c \b^{-4}(h_{22} n_{(1)2})]
\\&=[\b^{-3}(m_{[-1]1})\c \z_{N}^{-3}((h\c n)_{(0)})]\c S\b^{-3}((h\c n)_{(1)1})
\\&\quad\o_{H} S\b^{-2}(m_{[-1]2})\c[\z_{M}^{-2}(m_{[0]})\c \b^{-4}((h\c n)_{(1)2})]
\\&=\widetilde{\si}(\z_{M}(m) \B_{H}h\c n).
\end{align*}

Next, it is sufficient to consider the case of four-angle Hopf modules $M=H\o V$ and $N=H\o W$ with $(V,\z_{V}), (W, \z_{W})\in \mathcal{YD}^{H}_{H}$. In this case the $\widetilde{\si}$ is defined by the commutative diagram
$$\xymatrix{
  (H\o V)\o_{H}(H\o W) \ar[d]_{\widetilde{\si}} \ar[r]^-{\vp}      &  (H\o V\o W)\ar[d]^{H\o c_{V,W}}  \\
  (H\o W)\o_{H}(H\o V)  \ar[r]_-{\vp}               & H\o W\o V           }
  $$
where the $c_{V,W}$ denotes the braiding in Theorem \ref{T3.5}. For any $g, h\in H$, $v\in V$ and $w\in W$, we have
\begin{align*}
\widetilde{\si}(g&\o v\o_{H} h\o w)=\vp^{-1}(id\o \si)\vp(g\o v\o_{H} h\o w)
\\&=\vp^{-1}(id\o \si)(\b^{-1}(g)\b^{-2}(h_{1})\o \z_{V}^{-1}(v)\tl\b^{-2}(h_{2})\o w )
\\&=\vp^{-1}(\b^{-1}(g)\b^{-2}(h_{1})\o \z_{W}^{-1}(w_{(0)})\o \z_{V}^{-1}(\z_{V}^{-1}(v)\tl\b^{-2}(h_{2}))\tl \b^{-2}(w_{(1)}) )
\\&=\vp^{-1}(\b^{-1}(g)\b^{-2}(h_{1})\o \z_{W}^{-1}(w_{(0)})\o (\z_{V}^{-2}(v)\tl\b^{-3}(h_{2}))\tl \b^{-2}(w_{(1)}) )
\\&=[\b^{-1}(g )\b^{-2}(h_{1})\o\z_{W}^{-1} (w_{(0)})]\o_{H}1 \o\z_{V}^{-1}(v_{[0]})\tl \b^{-3}(h_{2})\b^{-3}(w_{(1) })
\\&=[\b^{-1}(g )\b^{-2}(h_{1})\o\z_{W}^{-1} (w_{(0)})]
\\&\quad \o_{H} S\b^{-4}( w_{(1)11})\b^{-4}(w_{(1)12})\o\z_{V}^{-1}(v_{[0]})\tl \b^{-3}(h_{2})\b^{-4}(w_{(1)2})
\\&=[\b^{-1}(g )\b^{-2}(h_{1})\o\z_{W}^{-1} (w_{(0)})]
\\&\quad \o_{H} S\b^{-4}( w_{(1)11})[[S\b^{-6}(h_{211})\b^{-6}(h_{212})]\b^{-5}(w_{(1)12})]\o\z_{V}^{-1}(v_{[0]})\tl \b^{-4}(h_{22}w_{(1)2})
\\&=[\b^{-1}(g )\b^{-2}(h_{1})\o\z_{W}^{-1} (w_{(0)})]
\\&\quad \o_{H} S\b^{-3}( w_{(1)1})[[S\b^{-5}(h_{21})\b^{-6}(h_{221}]\b^{-5}(w_{(1)21})]\o\z_{V}^{-1}(v_{[0]})\tl \b^{-5}(h_{222}w_{(1)22})
\\&=[\b^{-1}(g )\b^{-2}(h_{1})\o\z_{W}^{-1} (w_{(0)})]
\\&\quad \o_{H} S\b^{-4}(h_{21}w_{(1)1})\b^{-5}(h_{221}w_{(1)21})\o\z_{V}^{-1}(v_{[0]})\tl \b^{-5}(h_{222}w_{(1)22})
\\&=[\b^{-1}(g )\b^{-2}(h_{1})\o\z_{W}^{-1} (w_{(0)})]
\\&\quad \o_{H} S\b^{-4}(h_{21}w_{(1)1})\c [\b^{-5}(h_{221}w_{(1)21})\o\z_{V}^{-2}(v_{[0]})\tl \b^{-6}(h_{222}w_{(1)22})]
\\&=[\b^{-2}(g )\b^{-3}(h_{1})\o\z_{W}^{-2} (w_{(0)})]\c S\b^{-4}(h_{21}w_{(1)1})
\\&\quad \o_{H} \b^{-4}(h_{221}w_{(1)21})\o\z_{V}^{-1}(v_{[0]})\tl \b^{-5}(h_{222}w_{(1)22})
\\&=[\b^{-3}(g_{ 1})\b^{-3}(h_{1})\o\z_{W}^{-2} (w_{(0)})]\c S\b^{-4}(h_{21}w_{(1)1})
\\&\quad \o_{H} [S\b^{-4}(g_{21})\b^{-4}(g_{22})]\b^{-5}(h_{221}w_{(1)21})\o\z_{V}^{-1}(v_{[0]})\tl \b^{-5}(h_{222}w_{(1)22})
\\&=[\b^{-4}(g_{ 11})\b^{-3}(h_{1})\o\z_{W}^{-2} (w_{(0)})]\c S\b^{-4}(h_{21}w_{(1)1})
\\&\quad \o_{H} [S\b^{-4}(g_{12})\b^{-3}(g_{2})]\b^{-5}(h_{221}w_{(1)21})\o\z_{V}^{-1}(v_{[0]})\tl \b^{-5}(h_{222}w_{(1)22})
\\&=[\b^{-4}(g_{ 11})\b^{-3}(h_{1})\o\z_{W}^{-2} (w_{(0)})]\c S\b^{-4}(h_{21}w_{(1)1})
\\&\quad \o_{H} S\b^{-3}(g_{12})[\b^{-3}(g_{2})\b^{-6}(h_{221}w_{(1)21})]\o\z_{V}^{-1}(v_{[0]})\tl \b^{-5}(h_{222}w_{(1)22})
\\&=[\b^{-4}(g_{ 11})\b^{-3}(h_{1})\o\z_{W}^{-2} (w_{(0)})]\c S\b^{-4}(h_{21}w_{(1)1})
\\&\quad \o_{H} S\b^{-3}(g_{12})\c [\b^{-3}(g_{2})\b^{-6}(h_{221}w_{(1)21})\o\z_{V}^{-2}(v_{[0]})\tl \b^{-6}(h_{222}w_{(1)22})]
\\&=[\b^{-4}(g_{ 11})\c (\b^{-3}(h_{1})\o\z_{W}^{-3} (w_{(0)}))]\c S\b^{-4}(h_{21}w_{(1)1})
\\&\quad \o_{H} S\b^{-3}(g_{12})\c[(\b^{-3}(g_{2})\o\z_{V}^{-2}(v_{[0]}))\c \b^{-5}(h_{22}w_{(1)2})]
\\&=[\b^{-4}(g_{ 11})\c (\b^{-3}\o\z_{W}^{-3})(h_{1}\o w_{(0)})]\c S\b^{-4}(h_{21}w_{(1)1})
\\&\quad \o_{H} S\b^{-3}(g_{12})\c[(\b^{-3}\o\z_{V}^{-3})(g_{2}\o \z_{V}(v_{[0]}))\c \b^{-5}(h_{22}w_{(1)2}]
\\&=[\b^{-4}((g\o v)_{[-1]1})\c (\b^{-3}\o\z_{W}^{-3})((h\o w)_{(0)})]\c S\b^{-3}((h\o w)_{(1)1})
\\&\quad \o_{H} S\b^{-3}((g\o v)_{[-1]2})\c[(\b^{-3}\o\z_{V}^{-3})((g\o v)_{[0]})\c \b^{-4}((h\o w)_{(1)2})].
\end{align*}

This completes the proof.
\end{proof}
As a dual of the Theorem \ref{T2}, we can consider the Hom-cotensor product $\B_{H}$. The structure of monoidal category $(~\!^{H}_{H}\mathfrak{M}^{H}_{H}, \B_{H})$ has been given as in Proposition \ref{P1}. Let $(M,\z_{M}), (N,\z_{N})\in \!^{H}_{H}\mathfrak{M}^{H}_{H}$. Define the following linear map, for any $m\o n\in M\B_{H}N$:
\begin{align*}
\widehat{\si}:M\B_{H} N&\ra N\B_{H} M
\\m\o n&\m \b^{-4}(m_{(0)[-1]})S\b^{-3}(n_{[-1]})\c \z_{N}^{-3}(n_{[0](0)})
\\&\qquad \quad\B \z_{M}^{-3}(m_{(0)[0]})\c S\b^{-3}(m_{(1)})\b^{-4}(n_{[0](1)}),
\end{align*}

Note that the $\widehat{\si}$ is bijective with inverse $\widehat{\si}^{-1}(n\o m)=\z^{-3}_{M}(m_{[0](0)})\tl S^{-1}\b^{-3}(m_{[-1]})$ $\b^{-4}(n_{(0)[-1]})\B \b^{-4}(m_{[0](1)})S^{-1}\b^{-3}(n_{(1)})\tr \z^{-3}_{N}(n_{(0)[0]})$
\begin{theorem}
Let $(H,\b)$ be a Hom-Hopf algebra such that $\b$ and the antipode $S$ are bijective. Then $(~\!^{H}_{H}\mathfrak{M}^{H}_{H},\B_{H})$ is a braided monoidal category with a braiding $\widehat{\si}$.
\end{theorem}
The proof is similar to Theorem \ref{T2}.

\section*{Acknowledgements}The authors would like to thank the reviewer for helpful comments.

\section*{Funding}
This research is partially supported by the Natural Science Foundation of the Jiangsu Higher Education Institutions of China (No. 22KJB110019)
and the Scientific Research Foundation of Nanjing Institute of Technology (No. YKJ202219).

\end{document}